\documentclass[12pt, a4paper]{amsart}
\usepackage{amscd,amssymb,eucal,eufrak,mathrsfs,amsmath}
\usepackage[hidelinks]{hyperref} 

\usepackage{stmaryrd,enumitem,mathtools} 

\input xy
\xyoption{all}
\usepackage{epsfig}

\newtheorem{thm}{Theorem}[section]
\newtheorem*{thm*}{Theorem}
\newtheorem{cor}[thm]{Corollary}
\newtheorem{lem}[thm]{Lemma}
\newtheorem{prop}[thm]{Proposition}

\theoremstyle{definition}
\newtheorem{defin}[thm]{Definition}

\theoremstyle{remark}
\newtheorem{remark}[thm]{Remark}
\newtheorem{remarks}[thm]{Remarks}
\newtheorem{question}[thm]{Question}                                     
\newtheorem{example}[thm]{Example}
\newtheorem{examples}[thm]{Examples}
\numberwithin{equation}{section}
                                                                      
\newcommand{\nt}{\noindent}

\def\eps{{\varepsilon}}

\def\U{\mathcal{U}}

\def\eps{{\varepsilon}}

\newcommand{\g}{\gamma}

\newcommand{\A}{\mathcal{A}}

\newcommand{\sk}{\vskip 0.2cm}

\newcommand{\ov}{\overline}

\newcommand{\ben}{\begin{enumerate}}

\newcommand{\een}{\end{enumerate}}
\newcommand{\bit}{\begin{itemize}}

\newcommand{\eit}{\end{itemize}}

\def\C {{\mathbb C}}
\def\R {{\mathbb R}}
\def\N {{\mathbb N}}

\def\mfG {{\mathfrak G}}

\def\RUC{\operatorname{RUC}}
\def\LUC{\operatorname{LUC}}
\def\UC{\operatorname{UC}}
\def\Def{\operatorname{Def}}
\def\Defqf{\operatorname{Def_{qf}}}
\def\Gro{\operatorname{Gro}}

\def\tp{\operatorname{tp}}

\newcommand{\actson}{\curvearrowright}

\def\diam{{\mathrm{diam}}}
\def\Iso{\operatorname{Iso}}
\def\Aut{\operatorname{Aut}}

\def\Homeo{\operatorname{Homeo}}

\def\nbd {neighborhood }

\begin{document}

\title[Maximal equivariant compactification of metric structures]{Maximal equivariant compactification of the Urysohn spaces and other metric structures}

\author[Tom\'as Ibarluc{\'\i}a]{Tom\'as Ibarluc{\'\i}a}
\address{
  Universit{\'e} de Paris\\
  CNRS, Institut de Math{\'e}matiques de Jussieu--Paris Rive Gauche\\
  F-75013 Paris\\
  France}

\urladdr{https://webusers.imj-prg.fr/~tomas.ibarlucia}

\author[Michael Megrelishvili]{Michael Megrelishvili}
\address{Department of Mathematics,
	Bar-Ilan University, 52900 Ramat-Gan, Israel}
\email{megereli@math.biu.ac.il}
\urladdr{http://www.math.biu.ac.il/~megereli}



\thanks{The first author's research was partially supported by the ANR contract AGRUME (ANR-17-CE40-0026). The second author's research was supported by a grant of the Israel Science Foundation 1194/19}

\begin{abstract} 
	We study isometric $G$-spaces and the question of when their maximal equivariant compactification is the \textit{Gromov compactification} (meaning that it coincides with the compactification generated by the distance functions to points). Answering questions of Pestov, we show that this is the case for the Urysohn sphere and related spaces, but not for the unit sphere of the Gurarij space.

	We show that the maximal equivariant compactification of a separably categorical metric structure $M$ under the action of its automorphism group can be identified with the space $S_1(M)$ of 1-types over $M$, and is in particular metrizable. This provides a unified understanding of the previous and other examples. In particular, the maximal equivariant compactifications of the spheres of the Gurarij space and of the $L^p$ spaces are metrizable.

	We also prove a uniform version of Effros' Theorem for isometric actions of Roelcke precompact Polish groups.
\end{abstract}

\maketitle

\setcounter{tocdepth}{1}
\tableofcontents
 
\section{Introduction}

\subsection*{Background}
Before explaining the aim of the paper let us recall that a \emph{$G$-space} of a topological group $G$ is a topological space $X$ together with a continuous left action $G\actson X$. We will assume that the phase space $X$ is always Tychonoff, i.e., that $X$ can be topologically embedded into a compact Hausdorff space. An \emph{equivariant compactification} of a $G$-space $X$ is given by a compact Hausdorff $G$-space $K$ and a continuous $G$-equivariant map $\nu\colon X\to K$ with a dense image. The map $\nu$ need not be a topological embedding (or even injective); if it is a topological embedding, the compactification is said \emph{proper}.

For locally compact groups all $G$-spaces admit proper compactifications, as was established by de Vries \cite{Vr-loccom78}. However, this fails in general, as first shown by the second author \cite{Me-Ex88} (resolving a question of de Vries \cite{Vr-can75}), who built a Polish fan $X$ together with a Polish group $G\leq\Homeo(X)$ such that the system $G\actson X$ has no injective $G$-compactifications. Recently, and answering an old question of Smirnov, Pestov \cite{Pest-Smirnov} exhibited an extreme counterexample by constructing a countable metrizable group $G$ and a countable metrizable non-trivial $G$-space $X$ for which every equivariant compactification is \emph{trivial}, i.e., a singleton. The example is obtained by a clever iteration of the construction of \cite{Me-Ex88}. Pestov's paper ends with a discussion of several open questions. In this paper we address some of these and related questions (see Questions \ref{q:3} and \ref{q:Pest} below), concerning important examples of \emph{isometric} $G$-spaces with greatest $G$-compactifications which are \textit{small} and admit tractable descriptions.

We recall as well that if $X$ is a $G$-space, a continuous bounded function $f\colon X\to \R$ is \emph{right uniformly continuous (RUC)} if for every
$\varepsilon > 0$ there exists a neighborhood $V$ of the identity $e\in G$ such that $\sup_{x \in X} |f(vx)-f(x) | < \varepsilon$ for every
$v \in V$. The set $\RUC_G(X)$ of all right uniformly continuous functions on $X$ is a closed $G$-invariant subalgebra of $\operatorname{CB}(X)$---the algebra of real-valued, continuous, bounded functions on $X$, with the supremum norm, on which $G$ acts by the formula $gf(x)=f(g^{-1}x)$. There is a natural bijective correspondence between the equivariant compactifications of $X$ (up to equivalence) and the closed $G$-invariant subalgebras of $\RUC_G(X)$ that are \textit{unital} (containing the constants). Below by (sub)algebra we always mean a \textit{unital} subalgebra of $\operatorname{CB}(X)$. The compactification corresponding to a subalgebra $\A$ is given by the Stone--Gelfand space of maximal ideals of $\A$.

 In particular, the algebra $\RUC_G(X)$ corresponds to the \emph{greatest} (or \emph{maximal}) \emph{equivariant compactification} of $X$, which we denote by $\beta_G\colon X\to \beta_G X$, and which is characterized by the property that any other equivariant compactification of $X$ factors through $\beta_G$. Pestov's construction in \cite{Pest-Smirnov} gives thus a $G$-space $X$ for which $\beta_G X$ is a singleton, or, equivalently, for which $\RUC_G(X)$ is as small as possible, namely the algebra of constant functions on $X$.

As indicated before, we will study $G$-spaces in which the phase space $X$ is a metric space and the action $G\actson X$ is by \emph{isometries}. In this case, there is always a natural family of non-trivial RUC functions. Indeed, if we assume moreover that the metric $d$ on $X$ is bounded, every element $z\in X$ induces a bounded, continuous function
$$f_z\colon X\to \R,\ x\mapsto d(x,z)$$
that is right uniformly continuous. Let $\Gro(X)$ denote the closed algebra generated by the functions of this form (plus the constants). Then, as is easy to check, $\Gro(X)$ is a $G$-invariant subalgebra of $\RUC_G(X)$. Following \cite{Akin,Me-opit07,Me-c_018,Pest-Smirnov}, we call the  equivariant compactification associated to the subalgebra $\Gro(X)$ the \emph{Gromov compactification} of the isometric $G$-space $X$, and we denote it by $\g \colon X\to \g X$. In the case where the metric on $X$ is not bounded, we propose a definition for the Gromov compactification in Section~\ref{s:GromovComp}. It is clear that $\g X$ is non-trivial as long as $X$ is non-trivial; in fact, $\g$ is always proper. Hence, isometric $G$-spaces cannot provide examples with the property of Pestov's, but one may ask for examples of isometric systems with no compactifications \emph{above} $\g$.

\pagebreak
\begin{example} \label{ex:Stoyanov}\ 

As mentioned in \cite{Pest-Smirnov}, an elegant geometric example where we can understand the compactifications $\g$ and $\beta_G$ is provided by the \emph{unit sphere} of the separable infinite-dimensional (complex or real) Hilbert space, $$X=S_{\ell^2} \coloneqq\{v\in\ell^2:\|v\|=1\},$$
under the action of the whole unitary (or orthogonal) group $G=U(\ell^2)$ with the strong operator topology. Indeed, Stoyanov \cite{Sto,Sto2,DPS} proved that the greatest equivariant compactification of $S_{\ell^2}$ can be identified with the \emph{unit ball} of $\ell^2$ with the weak topology. From this, one can deduce moreover that $\g=\beta_G$ up to equivalence (see Proposition \ref{p:S-Gr} below).
\end{example}
 
\begin{remarks} \label{r:2} \ 
	\begin{enumerate}
		\item As the case of $\beta_G(S_{\ell^2})$ shows, the maximal $G$-compactification of a Polish non-compact space might be metrizable for dynamically massive actions. Recall, in contrast, that the \v{C}ech--Stone compactification $\beta X$ of any Polish non-compact space $X$ cannot be metrizable.
		\item Let $\beta_G G$ be the greatest $G$-compactification of 
		the standard left action of a topological group $G$ on itself (the so-called \textit{greatest ambit} of~$G$). Then
		$\beta_G G$ is metrizable if and only if $G$ is \textit{precompact} and second countable. On the other hand, there are interesting cases with metrizable 
		 $\beta_G X$ for Polish coset $G$-spaces $X=G/H$ (e.g., the unit sphere $S_{\ell^2}$ from Example \ref{ex:Stoyanov} and the Urysohn sphere $\mathbb{U}_1$ from Theorem \ref{t:Urys-sphere} and Example \ref{ex:6Examples}.4).
	\end{enumerate}
\end{remarks}

\sk

\subsection*{Problems and results}
The beautiful result of Stoyanov from Example \ref{ex:Stoyanov} motivates the following general questions and problems:

\begin{question} \label{q:3} \ 
\begin{itemize}
	\item [(a)] (Smirnov \cite{Sm-geom}) Can ``simple geometric objects'' be maximal equivariant compactifications?
	\item [(b)] Study the greatest $G$-compactification $\beta_G\colon X \to \beta_G X$ of (natural) Polish $G$-spaces. In particular: when is $\beta_G X$ metrizable ?
	\item [(c)] Study the Gromov compactification $\g\colon X \to \g X$ for natural isometric actions of Polish groups. In particular: when do we have $\g=\beta_G$ (up to equivalence)?
\end{itemize}
\end{question}

\begin{question} \label{q:Pest}\ 

More concretely, the question is raised by Pestov in \cite{Pest-Smirnov} as of whether the equation $\g = \beta_G$ holds in the following examples of $G$-spaces $X$ (which resemble the unit sphere of the Hilbert space in many aspects) :
\begin{enumerate}
	\item The \emph{Urysohn sphere} $X=\mathbb{U}_1$, under the action of the whole isometry group $G=\Iso(\mathbb{U}_1)$.
	\item The unit sphere $X=S_\mfG$ of the \emph{Gurarij space} $\mfG$, under the action of the linear isometry group  $G=\Aut(\mfG)$.
	\item The unit spheres of other distinguished Banach spaces under the action of the corresponding linear isometry groups. For the spaces $L^p[0,1]$, $1<p<\infty$, $p\neq 2$, is it true that the natural compactification $S_p\to B_p^w$ of the unit sphere into the unit ball with the weak topology is the maximal equivariant (or the Gromov) compactification of $S_p$?
\end{enumerate}
All the groups in question are endowed with the topology of pointwise convergence (i.e., for groups of linear isometries, the strong operator topology).
\end{question}

One of our main results
 is a positive answer in the case of the Urysohn sphere (see Theorem~\ref{t:Urys-sphere} and its alternative proof in Example \ref{ex:6Examples}.4). That is, the greatest equivariant compactification of  $\mathbb{U}_1$ is the Gromov compactification. After some preliminaries on the Gromov compactification in Section~\ref{s:GromovComp}, we give a direct proof of this result in Section~\ref{s:Urysohn}.  
  Actually, the proof uses a couple of key properties of 
  $\mathbb{U}_1$ and applies also to the unbounded Urysohn space $\mathbb{U}$, as well as to non-separable or non-complete Urysohn-like spaces.

In Section~\ref{s:Models}, after discussing a unified, model-theoretic approach to the examples mentioned above, we show that the answer is negative for the unit sphere of the Gurarij space (see Theorem \ref{t:Gurarij}). In other words, denoting $G=\Aut(\mfG)$, the algebra $\RUC_G(S_\mfG)$ is strictly larger than the closed algebra generated by the functions $f_z(x)=\|x-z\|$ for $z\in S_\mfG$. Nevertheless, we show that if $S_V$ is the unit sphere of a \emph{separably categorical}, \emph{approximately ultrahomogeneous} Banach space $V$ and $G=\Aut(V)$ is the corresponding linear isometry group, then $\RUC_G(S_V)$ is generated by the functions $f_v(x)=\|x-v\|$ with $v\in V$ (see Theorem~\ref{t:Banach-spaces}). That is, one needs to consider the distance functions to elements outside the unit sphere, but this is enough.

This result applies to the Gurarij space, but also to the Banach spaces $L^p[0,1]$ for $1\leq p<\infty$, $p\notin 2\N$. Moreover, our methods allow us to give a negative answer to Pestov's question concerning the natural compactification $S_p\to B_p^w$ of the sphere of $L^p[0,1]$ into the unit ball with the weak topology. More precisely, we show that for $1<p<\infty$, $p\neq 2$, the Gromov compactification is \emph{not} a factor of the weak unit ball. On the other hand, we do not know whether the maximal equivariant and the Gromov compactifications coincide.

Let us say some words about the model-theoretic approach of Section~\ref{s:Models}. The Urysohn sphere, the Gurarij space and the Banach spaces $L^p[0,1]$, $1\leq p<\infty$---when seen as structures in the appropriate languages, in the sense of continuous logic---are examples of \emph{separably categorical} structures. This means that they are the only separable models of their respective first-order theories, and implies a number of strong properties. In \cite{bentsa}, Ben Yaacov and Tsankov showed how to translate many properties of separably categorical structures into facts about the dynamics of their automorphism groups.

One consequence of the ideas of \cite{bentsa}, as stated and exploited in \cite{ibaDyn}, is that if $M$ is a separably categorical structure and $G$ is its automorphism group, then a function $f\in\RUC_G(M)$ can be seen as a \emph{definable predicate} of the structure $M$, \emph{provided} that $f$ is uniformly continuous with respect to the metric of $M$. One simple but crucial observation of the present paper, which had gone unnoticed before, is that the hypothesis of uniform continuity can be dropped. In other words, for separably categorical structures we have the equality:
\begin{equation*}
\RUC_G(M)=\Def(M).
\end{equation*}
Thus the maximal equivariant compactification $\beta_G M$ is precisely $S_1(M)$, the space of 1-types over $M$. (For the definitions of $\Def(M)$ and $S_1(M)$ see Section~\ref{s:Models}.) In particular, $\beta_G M$ is \textit{metrizable}.

If, moreover, $M$ enjoys quantifier elimination in a natural language, this permits to understand the compactification $\beta_G M$ fairly well, and helps to determine whether it coincides with the Gromov compactification. In fact, the condition $\Gro(M)=\Def(M)$ can be seen as a \emph{metric} form of the classical model-theoretic notion of \emph{minimality}: every one-dimensional definable predicate is a continuous combination of distance functions to points. With this approach we can easily (re-)prove that $\beta_G=\g$ for the unit sphere of the Hilbert space (recovering Stoyanov's result) and for the Urysohn sphere, as well as for other spaces such as the Rado graph with the graph metric (see Examples~\ref{ex:6Examples}). The result about separably categorical Banach spaces mentioned above is also an immediate consequence of this method.

Finally, in Section~\ref{s:UMT}, we study \emph{uniformly micro-transitive} $G$-spaces, a notion that is related to the topics of the preceding sections. We record some basic remarks and prove that every separably categorical, transitive structure is uniformly micro-transitive. This yields a uniform version of Effros' theorem for isometric actions of Roelcke precompact Polish groups (see Theorem \ref{t:UnifEffr}).

\subsection*{Acknowledgments} We are grateful to Vladimir Pestov for putting us in contact after knowing of our independent approaches to the questions of \cite{Pest-Smirnov}. The first author would like to thank Ita{\"\i} Ben Yaacov and Todor Tsankov for enriching conversations.

\bigskip

\noindent\hrulefill

\sk 
\section{The Gromov compactification of isometric systems} \label{s:GromovComp}
\sk

In this section we give a definition of the Gromov compactification for general (not necessarily bounded) metric spaces, and discuss some of its basic properties in connection with isometric actions.

Let $(X,d)$ be a metric space. We will denote by $\U_{w,d}$ the initial uniformity on $X$ generated by the set $X^d$ of \emph{elementary Kat\v{e}tov functions}:
$$X^d\coloneqq  \{f_z\colon  X \to \R,\ x \mapsto d(x,z): z \in X\}.$$
Equivalently, $\U_{w,d}$ is generated by the system $\{d_z: z \in X\}$ of pseudometrics defined by
$$
d_z(x,y)\coloneqq  |d(x,z)-d(y,z)|. 
$$
It is clear that $\U_{w,d}$ is coarser than $\U_d$, where $\U_d$ denotes the usual uniformity induced by the metric on $X$. It is also easy to check that $\U_{w,d}$ is compatible, i.e., it induces the same topology as $\U_d$. We call $\U_{w,d}$ the \emph{$d$-weak uniformity}. 

We recall that the \emph{precompact replica} of a Hausdorff uniformity $\U$ on a set $X$ is the finest uniformity among all totally bounded uniformities on $X$ that are coarser than~$\U$. The precompact replica of $\U$ is always compatible with $\U$, and its completion is called the \emph{Samuel compactification} of $(X,\U)$. See, for instance, \cite{Isbell}. The subalgebra of $\operatorname{CB}(X)$ that corresponds to the Samuel compactification of $(X,\U)$ is just the algebra of all bounded, $\U$-uniformly continuous, real-valued functions on $X$. 

\begin{defin} \label{d:Gr-comp}
	Let $(X,d)$ be an arbitrary metric space. We define the \emph{Gromov compactification} of $X$ as the Samuel compactification of $(X,\U_{w,d})$, and we denote it by $\g\colon X\to \g X$.
\end{defin}

Note, in particular, that $\gamma$ is always a topological embedding. If $(X,d)$ is separable and bounded then $\g X$ is a metrizable compactum.

We will denote by $\A_{w,d}$ the algebra of real-valued continuous functions on $X$ that factor continuously through the Gromov compactification of $X$. Thus, $\A_{w,d}$ is precisely the algebra of $\U_{w,d}$-uniformly continuous, bounded, real functions on $X$.

Now suppose that $G$ is a topological group acting continuously and by isometries on $X$, i.e., $X$ is an isometric $G$-space. We will argue that in that case $\g X$ carries the structure of a $G$-space as well, so that the map $\gamma$ becomes an equivariant compactification of the system $G\actson X$.

For this it suffices to check that the algebra $\A_{w,d}$ is $G$-invariant and contained in $\RUC_G(X)$. Equivalently, we can prove that $\U_{w,d}$ is an \emph{equiuniformity} on the $G$-space~$X$. We recall that if $X$ is an arbitrary $G$-space, an equiuniformity $\U$ on $X$ is a compatible uniformity such that: 
\begin{enumerate}
	\item ($\U$ is \emph{saturated}) for every $g\in G$, the translation $g \colon X \to X$ is $\U$-uniform; 
	\item ($\U$ is \emph{motion equicontinuous}, \cite{Br}) for every entourage $\varepsilon \in
	\U$ there exists a neighborhood $U\in N_e$ of the identity such that
	$(gx,x) \in \varepsilon$ for every $(g,x) \in U \times X$.
\end{enumerate}

\begin{prop}
	[Brook \cite{Br}]
	\label{p:equiuni-Samuel}
Let $X$ be a $G$-space and $\U$ be an equiuniformity 
on~$X$. Then the action $G\actson X$ extends to a continuous action on the Samuel compactification of $(X,\U)$.
\end{prop}
\begin{proof}
Let $\A\subseteq \operatorname{CB}(X)$ be the subalgebra of $\U$-uniformly continuous functions on $X$. Saturation shows that $\A$ is $G$-invariant, and motion equicontinuity implies that every $f\in \A$ is RUC.
\end{proof}

As said in the introduction, for a \emph{bounded} metric space $(X,d)$ we will denote by $\Gro(X)$ the closed (unital) subalgebra of $\operatorname{CB}(X)$ generated by the set $X^d$ of elementary Kat\v{e}tov functions.

We recall that an arbitrary function $\xi\colon X\to\R$ is \emph{Kat\v{e}tov} if 
$$
|\xi(x)-\xi(y)| \leq d(x,y) \leq \xi(x)+\xi(y) \ \ \ \ \forall x,y \in X. 
$$ 
We will denote by $\operatorname{K}(X)$ the set of Kat\v{e}tov functions on $X$ that are bounded by the diameter of $X$. Thus, if $X$ is bounded, $\operatorname{K}(X)$ is a compact space with the topology of pointwise convergence.

\begin{prop} \label{p:Gr-comp}  Let $(X,d)$ be a metric space, $G$ a topological group and $G\actson X$ a continuous isometric action. Then:
	\begin{enumerate}
		\item The $d$-weak uniformity $\U_{w,d}$ is an equiuniformity of $X$ as a $G$-space, and the compactification $\g\colon X \to \g X$ is a proper equivariant compactification.
		\item Suppose the metric $d$ is bounded.
		\begin{enumerate}
		\item Then $\A_{w,d}=\Gro(X)$. In particular, $\g=\beta_G$ (up to equivalence) if and only if $\RUC_G(X)$ is generated by $X^d$ (as a closed unital algebra).
		\item The space $\g X$ can be identified with the closure of $X^d$ inside $\operatorname{K}(X)$, and $\gamma$ with the map sending $z$ to the elementary Kat\v{e}tov function $f_z$.
		\end{enumerate}
	\end{enumerate} 
\end{prop} 
\begin{proof} (1) The equation $d_z(gx,gy)=d_{g^{-1}z}(x,y)$ shows that each $g\colon X\to X$ is uniformly continuous with respect to the system of pseudometrics $\{d_z:z\in X\}$. That is, $\U_{w,d}$ is saturated. On the other hand, given $z\in X$ and $\epsilon>0$, we can find a \nbd $U\in N_e$ such that $d(g^{-1}z,z)<\epsilon$ for every $g\in U$. From this and the inequality
	 $$d_z(gx,x)=|d(gx,z)-d(x,z)|=|d(x,g^{-1}z)-d(x,z)|\leq d(g^{-1}z,z),$$ we see that $\U_{w,d}$ is motion equicontinuous. Thus $\U_{w,d}$ is an equiuniformity and Proposition~\ref{p:equiuni-Samuel} implies that $\gamma\colon X\to\g X$ is a proper equivariant compactification of $G\actson X$.

	\vskip 0.2cm
	(2.a)  
	If $d$ is bounded then every elementary Kat\v{e}tov function $f_z$ is bounded. 
	The algebra $\Gro(X,d)$ and the larger algebra $\A_{w,d}$ induce the same compatible precompact uniformity on $X$  
	(namely, the precompact uniformity $\U_{w,d}$ on $X$ induced by the set of real functions $X^d\coloneqq  \{f_z\colon  X \to \R: z \in X\}$). This implies that $\Gro(X,d)=\A_{w,d}$.
	
	(2.b) 
	Under the identification $X\simeq X^d$, $z\mapsto f_z$, the $d$-weak uniformity on $X$ is precisely the trace of the compact uniformity of $\operatorname{K}(X)$ on $X^d$. Indeed, each subbasic entourage $\{(x,y)\in X^2:d_z(x,y)<\eps\}$ is the restriction of the subbasic entourage $\{(\xi,\zeta)\in\operatorname{K}(X)^2:|\xi(z)-\zeta(z)|<\eps\}$. Hence we may identify the closure $\ov{X^d}\subseteq\operatorname{K}(X)$ with the completion of $(X,\U_{w,d})$, and thus with the Gromov compactification of~$X$.
\end{proof}

We observe that, as a corollary, we can derive a result of Ludescher and de Vries \cite{Lud-Vr80}, which asserts that if a $G$-space $X$ admits a $G$-invariant metric then $X$ admits a proper equivariant compactification.

\begin{remark} \label{r:nonMetr}\ 

	Let $\U_{w,d}$ be the $d$-weak uniformity of $(X,d)$, where $d$ is unbounded. Then $\U_{w,d}$ is not totally bounded. Hence, the corresponding Samuel compactification is not metrizable. Indeed, if a uniform space $\U$ is not totally bounded, it contains a uniformly discrete infinite subset. It is then easy to see that the Samuel compactification of $\U$ contains a subspace homeomorphic to $\beta \N$. 
\end{remark}

Recall that the \emph{Roelcke uniformity} of a topological group $G$ is the intersection $\U_L\cap \U_R$ of the natural left and right uniformities of $G$ (for first-countable groups, the left and right uniformities are induced, respectively, by any left- or right-invariant compatible metrics on~$G$). The \emph{Roelcke compactification} of $G$ is the Samuel compactification of the Roelcke uniformity \cite{Usp-comp}.

\begin{prop}
Let $(X,d)$ be a metric space and $G\actson X$ a continuous isometric action. Given $a \in X$, let $\Gamma_a$ be the closure of $Ga$ in $\g X$, and let $\g_a\colon  G \to \Gamma_a,\ g\mapsto ga$ be the induced $G$-ambit. Then $\g_a$ is a factor of the Roelcke compactification of $G$.
\end{prop} 
\begin{proof} It is enough to show that the orbit map 
	$$\tilde{a}\colon  G \to (X,\U_{w,d}),\ g\mapsto ga$$ is left and right uniformly continuous. Right uniform continuity follows from the fact that $\U_{w,d}$ is motion equicontinuous. Now we show that $\tilde{a}\colon  (G,\U_L) \to (X,\U_{w,d})$ is uniform. It is equivalent to show that $f_z \circ \tilde{a}\colon  (G,\U_L) \to \R$ is uniformly continuous for every $z \in X$. 
	Given $\eps >0$ choose $U\in N_e$ such that $d(a,ua) < \eps \ \forall u \in U$. Then for every $(g,u) \in G \times U$ we have
	 $$|f_z(ga)-f_z(gua)| =|d(ga,z)-d(gua,z)|= |d(a,g^{-1}z) - d(ua,g^{-1}z)| \leq
	d(a,ua) < \eps,$$
proving the uniform continuity with respect to $\U_L$.
\end{proof}

\sk

Now let us recall a few facts about \textit{proximity relations} (see, for example, \cite{NW}) that will be used in the next section. If $\U$ is a uniformity on a space $X$ (which we think of as a system of entourages on $X$), the \emph{proximity relation associated to $\U$} is a binary relation $\delta_\U$ between subsets of $X$, defined by:
$$A\delta_\U B\iff \forall \eps\in\U, \ \ (A\times B) \cap \eps\neq\emptyset.$$
For instance, if $(X,d)$ is a metric space and $\delta_{w,d}$ denotes the proximity relation associated to the $d$-weak uniformity on $X$, then for any $A,B\subseteq X$ we have $A\delta_{w,d}B$ if and only if for every $\eps>0$ and every finite set $F\subseteq X$ there exist $a\in A$ and $b\in B$ such that $d_z(a,b)<\eps$ for each $z\in F$.

If $\U_1,\U_2$ are two uniformities on a space $X$ such that for every $A,B\subseteq X$ the relation $A\delta_{\U_1}B$ implies $A\delta_{\U_2}B$, then the precompact replica of $\U_2$ is coarser than the precompact replica of $\U_1$. In particular, if $\U_2$ is totally bounded, then $\U_2\subseteq \U_1$. See, for instance, \cite[Ch.~2, Thm.~35]{Isbell}.

\begin{remark} \label{r:prox-delta-G-X}\ 

Let $X$ be a $G$-space and let $\beta_G\colon X\to \beta_G X$ be the corresponding greatest equivariant compactification. Let $\U_G$ denote the trace on $X$ of the unique compatible uniformity on the compact space $\beta_G X$. A family of basic entourages of $\U_G$ is given by the sets of the form
$$\{(x,y)\in X\times X:\forall f\in F,\ |f(x)-f(y)|<\eps\}$$
where $\eps>0$ and $F\subseteq \RUC_G(X)$ is a finite subset. Let $\delta_G$ denote the proximity relation associated to $\U_G$. Then for every $A,B\subseteq X$ we have:
$$\big(\forall U\in N_e,\ \ov{UA}\cap\ov{UB}\neq\emptyset\big) \implies A\delta_G B.$$
This follows easily from the definition of $\RUC$ functions. 
\end{remark}

\sk

We end this section with a characterization of the Gromov compactification for the spheres in Hilbert spaces.

\begin{prop} \label{p:S-Gr} For every Hilbert space $H$ and its sphere $S_H$ 
	the compactification $\nu\colon  S_H \to B^w_H$ (where $B^w_H$ is the unit ball of $H$ endowed with the weak topology) is equivalent to the Gromov compactification of $S_H$. 
\end{prop}
\begin{proof} Let $H \times H \to H, (u,v) \mapsto \langle u, v\rangle$ denote the inner product of the Hilbert space. 
	For each given vector $z \in H$  define the pseudometric
	$$
	\rho_{z}(u,v) = |\operatorname{Re}\langle u,z\rangle - \operatorname{Re}\langle v,z\rangle|. 
	$$
	 The set $\{\rho_{z}: z \in S_H\}$ is a uniform subbase of the \emph{weak uniformity} $\mathcal{U}_w$ on $H$ which induces the weak topology on $H$.  
	The compactification $\nu\colon  S_H \to B^w_H$ is the completion of $S_H$ with respect to the  uniformity $\mathcal{U}_w|_{S_H}$. 
	
	On the other hand we have the precompact uniformity $\U_{w,d}$ of the Gromov compactification of $(S_H,d)$ with a uniform subbase generated by the system of pseudometrics $\{d_{z}: z \in S_H\}$, where 
	$$
	d_{z}(u,v)= \big| \|u-z\| -\|v-z\| \big|. 
	$$
	Now observe that $d_z$ and $\rho_z$ are uniformly equivalent on $S_H$. 
	Indeed, for $z,u,v \in S_H$ we have:
	$$\rho_z(u,v)=\big|(1-\operatorname{Re}\langle u,z\rangle) - (1-\operatorname{Re}\langle v,z\rangle)\big| = \frac{1}{2}\big|\|u-z\|^2-\|v-z\|^2\big| \leq 2d_z(u,v).$$
	Conversely:
	$$d_z(u,v)^2\leq \big| \|u-z\| -\|v-z\| \big|\cdot \big| \|u-z\| +\|v-z\| \big| = \big|\|u-z\|^2-\|v-z\|^2\big| = 2\rho_z(u,v).$$
	This proves that $\U_w|_{S_H} = \U_{w,d}$.
	\end{proof}

\noindent\hrulefill

\sk 
\section{Equivariant compactifications of Urysohn-like spaces}\label{s:Urysohn}
\sk

We begin with the definition of two auxiliary notions, the first of which will be investigated further in Section~\ref{s:UMT}.

Below, $B_\delta(x)$ denotes the ball $\{y\in X:d(x,y)<\delta\}$ and $N_e$ stands, as before, for the set of 
open neighborhoods of the identity in a group $G$.

 \begin{defin} \label{d:UWMT} 
 	 Let $G$ be a topological group and $(X,d)$ be a metric space. Given an action $G\actson X$, we say the action is: 
 	\begin{enumerate}
 	\item \emph{Uniformly weakly micro-transitive (UWMT)} if for every $U \in N_e$ there is $\delta>0$ such that $B_\delta(x)\subseteq \ov{Ux}$ for all $x\in X$. 		
 	\item \emph{Metrically achievable (MA)} if for every $\eps >0$ and $U \in N_e$ there exist a finite subset $F \subseteq X$ and $\delta >0$ such that, for every $x,y\in X$, 
 	\begin{equation*}
 	\big(\forall z \in F,\ |d(z,x)-d(z,y)| < \delta\big) \implies (\exists g \in U,\ d(gx,y) < \eps).
 	\end{equation*}
 	\end{enumerate}
 \end{defin}

 \sk 
 
 \begin{thm} \label{t:coincidence} Let $(X,d)$ be a metric space and $G\actson X$ be a continuous isometric action of a topological group $G$. Suppose that the action is 
 	(UWMT) and (MA). Then the greatest $G$-compactification of $X$ is just the Gromov compactification of $(X,d)$ (that is, $\beta_G=\g$ up to equivalence). 
\end{thm} 
\begin{proof} 
We have to show that $\beta_G$ is a factor of $\g$. As in Remark~\ref{r:prox-delta-G-X}, let $\delta_G$ be the proximity relation associated to the uniformity $\U_G$ induced on $X$ by the greatest equivariant compactification $\beta_G$. Let, on the other hand,  $\delta_{w,d}$ be the proximity relation associated to the $d$-weak uniformity $\U_{w,d}$. It is enough to show that
$$A \delta_{w,d} B \implies A \delta_G B$$
for every pair of subsets $A,B\subseteq X$. Indeed, since $\U_G$ is totally bounded, this implies that $\U_G$ is coarser than the precompact replica of $\U_{w,d}$. Hence there is a continuous map from $\g X$ (the Samuel compactification of $\U_{w,d}$) to $\beta_G X$ which is the identity on~$X$. This gives the desired factor map.

So suppose that $A\delta_{w,d} B$.
	\sk
\nt	\textbf{Claim 1:} $\forall U \in N_e,\ d(UA,B) =0$.                                                                                                                                                                                                                                                                                                                                                                                                                                                                                                                                                                                                                                                                                                                                                                                                                                                                                                                                                                                                                                                                                                                                                                                                                                                                                                                                                                                                                                                                                                                                                                                                                                                                                                                                                                                                                                                                                                                                                         
	\sk 
	
	Proof: Let $\eps >0$ and $U \in N_e$, and choose a corresponding finite set $F\subseteq X$ and $\delta>0$ as given by property (MA) . Since $A\delta_{w,d}B$, there exist $a\in A$, $b\in B$ such that $|d(z,a)-d(z,b)| <\delta$ for every $z\in F$. Hence, by (MA), there is $g \in U$ such that  $d(ga,b) < \eps$.
	This proves Claim 1. 
	
	\sk
\nt	\textbf{Claim 2:} $A \beta_G B$. 
	\sk 
	
	Proof: If not, then as per Remark~\ref{r:prox-delta-G-X} there exists $U \in N_e$ such that $\ov{UA} \cap \ov{UB} =\emptyset$.
	By (UWMT) there is $\delta >0$ such that $B_\delta(x)\subseteq \ov{Ux}$ for all $x\in B$. In particular, for every $b \in B$ we have 
	$UA \cap B_\delta(b)=\emptyset$. Hence $d(UA,B) \geq \delta$, contradicting Claim~1.
	\end{proof}

\sk 

Let us now recall some standard definitions and variants. Here we consider only metric spaces; see Definition~\ref{d:app-ultrahomogeneity} for the case of richer structures.

\begin{defin}[See, for example, \cite{Usp-sub,Pe-nbook,Mel-07}]
	Let $(X,d)$ be a metric space. 
	\begin{enumerate}
		\item $(X,d)$ is \emph{ultrahomogeneous} if every partial isometry $p\colon A \to B$ between finite subsets $A,B\subseteq X$ can be extended to an isometry $g\colon X \to X$.  
		\item Let $G\leq\Iso(X)$ be a subgroup of the group of isometries of $X$. We will say that $(X,d)$ is \emph{approximately $G$-ultrahomogeneous} if for every $\eps >0$ and every partial isometry $p\colon A \to B$ between finite subsets $A,B\subseteq X$ there exists $g \in G$ such that $d(pa,ga) < \eps$ for all $a \in A$.
		\item Let $\diam(X)\in [0,\infty]$ be the diameter of $(X,d)$. We say $(X,d)$ is \emph{finitely injective} if for every pair of finite metric spaces $K \subseteq L$ with diameter less or equal to $\diam(X)$, and every
		isometric embedding $\phi\colon  K \to X$, there exists an isometric embedding $\Phi\colon  L\to X$ that extends $\phi$. In the literature this is also called the \emph{one-point extension property}, since it is enough to check it for $L$ being a one-point extension of $K$.
	\end{enumerate}
\end{defin}

\begin{remark}\ 

Every (not necessarily separable) metric space with diameter $\leq 1$ can be isometrically embedded into a finitely injective ultrahomogeneous metric space (with the same topological weight) with diameter $\leq 1$ ; see Uspenskij \cite[Thm.~5.1]{Usp-sub}.
\end{remark}

\sk
 
 \begin{thm} \label{t:abstract-case}  
 	Let $(X,d)$ be a metric space and $G\leq\Iso(X)$ be a group of isometries, endowed with the topology of pointwise convergence. Suppose $(X,d)$ is finitely injective and approximately $G$-ultrahomogeneous. Then the greatest equivariant compactification of the $G$-space $X$ is the Gromov compactification of $(X,d)$.
 \end{thm}
 \begin{proof} 
 	By Theorem \ref{t:coincidence} it is enough to show that the natural action $G\actson X$ is (MA) and (UWMT). 
 	We first show that the action is (MA)  (Definition \ref{d:UWMT}). Let 
 	$$U=\{g \in G: d(ga,a) < \eps \ \forall a \in A\}\in N_e$$ be the \nbd determined by some $0<\eps<\diam(X)$ and some finite subset $A\subseteq X$.  It is enough to find $\delta >0$ and a finite subset $F \subseteq X$ such that  if 
 	\begin{itemize}[leftmargin=35pt]
 	\item [(a)] $|d(x,z)-d(y,z)| < \delta\ \ \forall z \in F$
 	\end{itemize} 
 	 then there exists $g \in G$ such that $d(gy,x) < \eps  \ \text{and} \ d(gz,z) < \eps \ \ \forall z \in F$.
 	
 	We choose $F=A$ and any $0<\delta < \eps/2$. Let $x,y\in X$ satisfy (a). We can suppose, in addition, that
 	\begin{itemize}[leftmargin=35pt]
 	\item [(b)] $\delta \leq d(x,z)+d(y,z) \ \ \forall z \in F$.
 	\end{itemize}
 	Indeed, otherwise $d(x,y) < \eps$ and we simply take $g=e$.
 	
 	\sk 
 	Let $K=F\cup\{x\}$. We consider an expansion $L=K\cup\{y'\}$ of the finite metric space $K$ by a new point $y'$ such that:
 	\begin{equation*}
 	\begin{cases}
 	d(y',x)=\delta\\
 	d(y',z)=d(y,z) \ \ \forall z\in F.
 	\end{cases}
 	\end{equation*}
To see that $L$ is a metric space, it suffices to check the triangle inequalities:
	$$
 	|d(y',x)-d(y',z)| \leq d(x,z) \leq d(y',x)+d(y',z) \ \ \ \forall z \in F. 
 	$$
The right-hand side inequality is equivalent to $d(x,z) - d(y,z) \leq  \delta$, which is true by our assumption (a). 
 	As to the left-hand side inequality, its one half is again just (a). The second half can be written as $\delta - d(y,z) \leq d(x,z)$ which is true by (b). 
 	
 	Since $(X,d)$ is finitely injective and $\diam(L)\leq\diam(X)$, we can assume that $y'$ belongs to $X$. Now, the map 
 	$$p\colon F \cup \{y\} \to F \cup \{y'\}, z \mapsto z, y \mapsto y'$$ is a partial isometry of $X$. Since $X$ is approximately $G$-ultrahomogeneous, there exists $g \in G$ that extends $p$ up to $\eps/2$.  Hence $d(gy,x) < d(gy,y')+d(y',x)<\epsilon/2+\delta< \eps$ and $d(gz,z)<\eps/2 \ \ \forall z \in F$, so $g$ is as desired.
 	
 	\sk \sk 
 	Next we prove that the action is (UWMT). We have to show that 
 	for every $\epsilon>0$ and finite subset $F \subseteq X$ there is $\delta>0$ such that for every $x,y \in X$ with $d(x,y)<\delta$ and every $\eta>0$ there exists $g \in G$ with the property that $d(gx,y)<\eta$ and $d(gz,z)<\epsilon$ for every $z\in F$. 
 	
 	We claim that we may choose $\delta=\eps/2.$ Indeed, 
 	let $F=\{z_1,\dots,z_k\}$. 
 	For convenience we denote $z_0=x$, $z'_0=y$. We expand the finite metric space $K= \{z_0,z'_0,z_1,\dots,z_k\}$ with $k$ new points $z_1',\dots,z_k'$, such that  
 	\begin{equation*}
 	\begin{cases}
 	d(z'_i,z'_j)=d(z_i,z_j)\\
 	d(z_i,z'_j)=\min\{d(z_i,z_j)+d(x,y), \diam(X)\} 
 	\end{cases}
 	\end{equation*}
 	for every $0\leq i,j\leq k$. This defines indeed a finite metric space $L$ of diameter less than $\diam(X)$ which contains $K$ as a metric subspace. Since $(X,d)$ is finitely injective, we may assume the points $z'_i$ exist in $X$.

 	Take $x,y\in X$ with $d(x,y)<\delta=\eps/2$ and let $\eta>0$. By approximate ultrahomogeneity, there exists  $g \in G$ which extends the partial isometry 
 	$$p\colon \{z_0,z_1,\dots,z_k\} \to \{z'_0,z_1',\dots,z_k'\}$$ up to $\min\{\eps/2,\eta\}$. In particular, we have $$d(gz_i,z_i)\leq d(gz_i,z'_i)+d(z'_i,z_i)<\eps$$ for all $i=1,\dots,k$. In addition, $d(gx,y) = d(gz_0,z_0')< \eta$. Hence $g$ is as required.
 \end{proof}	

\sk 

The \emph{Urysohn space}, $\mathbb{U}$, is the unique (up to isometry) Polish, finitely injective metric space of infinite diameter. It is ultrahomogeneous and universal for Polish metric spaces (i.e., every Polish metric space embeds in $\mathbb{U}$), and in fact is also characterized up to isometry by the conjunction of these two properties; see \cite{Usp-sub,Urys}. By a result of Uspenskij \cite{Usp90},
the isometry group $\Iso(\mathbb{U})$, endowed with the topology of pointwise convergence, is a universal Polish topological group.

The diameter 1 version of the Urysohn space, the \emph{Urysohn sphere}, $\mathbb{U}_1$, is the unique Polish, finitely injective metric space of diameter 1. It is also characterized by being ultrahomogeneous and universal for Polish (or just finite) metric spaces of diameter~1. Its isometry group $\Iso(\mathbb{U}_1)$, with the pointwise convergence topology, is a universal Polish group which is moreover Roelcke precompact; see \cite{Usp-sub}.

The Urysohn sphere and the Urysohn space are the main particular 
cases of Theorem~\ref{t:abstract-case} (because these metric spaces are finitely injective and ultrahomogeneous). 

\begin{thm} \label{t:Urys-sphere} 
	Let $\mathbb{U}_1$ be the Urysohn sphere.  
	Then the greatest equivariant compactification of the $G$-space $\mathbb{U}_1$ 	with $G=\Iso(\mathbb{U}_1)$ 	is the Gromov compactification of $\mathbb{U}_1$. In particular, $\beta_G(\mathbb{U}_1)$ is metrizable and the algebra $\RUC_G (\mathbb{U}_1)$ is just the closed algebra generated by the set 
$$\mathbb{U}_1^d=\{d(\cdot,z)\colon  \mathbb{U}_1 \to [0,1]: \  z \in \mathbb{U}_1\}$$
of elementary Kat\v{e}tov functions on $\mathbb{U}_1$. 
\end{thm}

By Effros' theorem (see Section \ref{s:UMT}) the $G$-space $\mathbb{U}_1$ can be identified with the coset $G$-space $G/H$ 
where $H=\operatorname{St}(a)$ is a stabilizer subgroup.   
As a corollary of Theorem \ref{t:Urys-sphere} we get that every bounded right uniformly continuous function $f\colon  \mathbb{U}_1= G / H \to \R$ on the coset $G$-space $G/H$  can be uniformly approximated by linear combinations of finite products of elementary Kat\v{e}tov functions from $\mathbb{U}_1^d$ together with the constants. 

\sk 
\begin{thm} \label{t:Urys-space}
	Let $\mathbb{U}$ be the Urysohn space.  
	Then the greatest equivariant compactification of the $G$-space $\mathbb{U}$ 
	with $G=\Iso(\mathbb{U})$ 
	is the Gromov compactification of $\mathbb{U}$, which is not metrizable
	(by Remark \ref{r:nonMetr}).  
\end{thm}

\sk 

We end this section with an observation about arbitrary equivariant compactifications of Urysohn-like spaces.

\begin{prop}\label{p:injective}
Let $(X,d)$ and $G$ be as in Theorem~\ref{t:abstract-case}. Then for every non-trivial $G$-equivariant compactification $\nu\colon X\to K$, the map $\nu$ is injective.
\end{prop}
\begin{proof}
Let $\nu\colon X\to K$ be a $G$-equivariant compactification, and suppose there are $x,y\in\mathbb{U}_1$ with $d(x,y)=r>0$ and $\nu(x)=\nu(y)$. By approximate $G$-ultrahomogeneity and equivariance, we have $\nu(x')=\nu(y')$ for every two points $x',y'\in \mathbb{U}_1$ at distance $r$. Now let $0\leq s\leq \diam(X)$. It is easy to construct a finite metric space $\{x_0,\dots,x_n\}$ such that $d(x_i,x_{i+1})=r$ for each $i<n$ and $d(x_n,x_0)=s$. By finite injectivity, we may assume that the $x_i$ are elements of $X$. It follows that $\nu(x_i)=\nu(x_{i+1})$ for each $i<n$, and thus $\nu(x_0)=\nu(x_n)$. As before, this implies that any two points at distance $s$ have equal image under $\nu$. We conclude that the compactification is trivial.
\end{proof}

On the other hand, there are \emph{non-proper} compactifications of $\mathbb{U}_1$; see Example~\ref{ex:6Examples}.4.

\begin{question}\ 

Describe all $G$-compactifications of the Urysohn sphere $\mathbb{U}_1$, where $G=\Iso(\mathbb{U}_1)$. 
\end{question}

\noindent\hrulefill

\sk
\section{A common approach for separably categorical structures}\label{s:Models}
\sk

In this section we work in the setting of continuous logic, as presented in \cite{BenUsv-local-stability} or~\cite{BY-et-al}. For ease of exposition we will consider only \emph{relational} languages, but everything below holds in the general case. A~\emph{metric structure} is thus a complete, bounded metric space $M$ together with distinguished (bounded, uniformly continuous) real-valued predicates---which are the interpretations for the symbols of the given language. The \emph{definable predicates} of the structure $M$---which correspond to interpretations of \emph{formulas}---are functions $M^n\to\mathbb R$ obtained from the basic predicates and the metric through continuous combinations, quantification (i.e., taking suprema or infima with respect to given variables) and uniform limits. We allow the definable predicates to depend on infinitely many variables, that is, $n\leq\omega$. The formalism of \cite{BenUsv-local-stability} requires that $M$ has diameter 1 and that all definable predicates take values in the interval $[0,1]$; but this is not necessary, and for our purposes we prefer the formalism of \cite{BY-et-al}, which is more general in this respect. On the other hand, a definable predicate of $M$ is always bounded and uniformly continuous with respect to the product uniformity on $M^n$ (which is the one induced, for instance, by the distance $d(x,y)=\sum_{i<n}2^{-i}d(x_i,y_i)$).

Let $M$ be a metric structure in this sense, and let $G=\Aut(M)$ be its automorphism group, i.e., the group of invertible isometries of $M$ that preserve the distinguished predicates. We endow $G$ with the topology of pointwise convergence induced by its natural action on $M$. Thus $G$ acts continuously on $M$, and also on the powers $M^n$ via the diagonal action. If $\varphi\colon M^n\to\mathbb R$ is any definable predicate, then $\varphi(ga)=\varphi(a)$ for every $g\in\Aut(G)$ and $a\in M^n$.

Henceforth we see $M$ as an isometric $G$-space. We consider the following closed $G$-invariant subalgebras of $\RUC_G(M)$:
\begin{itemize}[leftmargin=15pt]
\item The algebra $\Gro(M)$ generated by the functions $x\mapsto d(x,b)$ for $b\in M$ (we recall that the metric on $M$ is bounded).

\item The algebra $\Def(M)$ of functions $x\mapsto\varphi(x,b)$ where $\varphi\colon M\times M^n\to\mathbb R$ is a definable predicate and $b\in M^n$ is any tuple ($n \leq\omega$).

\item The algebra $\RUC^u_G(M)$ of right uniformly continuous functions on $M$ that are moreover uniformly continuous with respect to the metric of $M$.
\end{itemize}
As is easy to see, we always have the inclusions:
$$\Gro(M)\subseteq \Def(M)\subseteq \RUC^u_G(M)\subseteq \RUC_G(M).$$
Hence we can refine the question of whether $\g^M=\beta_G^M$ (i.e., $\Gro(M)=\RUC_G(M)$) for a given structure $M$ by asking whether the equality holds in each of the former inclusions. Moreover, a number of model-theoretic tools are available for the description of the algebra $\Def(M)$ in many concrete, interesting cases.

This analysis is particularly useful in the case of \emph{separably categorical} structures. We recall that $M$ is separably categorical (or \emph{$\aleph_0$-categorical}) if it is the unique separable model of its first-order theory, up to isomorphism. When considering $\aleph_0$-categorical structures we always assume that the language is countable (possibly up to interdefinability). The following was observed in \cite[Prop.~1.7]{ibaDyn}.

\begin{prop}\label{p:DEF=RUCu}
If $M$ is $\aleph_0$-categorical, then $\Def(M)= \RUC^u_G(M)$.
\end{prop}

This result was based on the ideas of \cite{bentsa}, where it is implicitly shown that in the $\aleph_0$-categorical setting every bounded \emph{Roelcke uniformly continuous} function $f\colon G\to\mathbb{R}$ (see Section~\ref{s:GromovComp}) can be represented in the form $f(g)=\varphi(a,gb)$ for an appropriate formula $\varphi(x,y)$ and tuples $a,b$ (possibly infinite). (See also \cite[Prop.~1.8]{ibaDyn}.) We denote by $\UC(G)$ the algebra of bounded Roelcke uniformly continuous functions on~$G$, and we recall that $f\in\UC(G)$ if and only if $f$ is both left and right uniformly continuous, i.e., $\UC(G)=\LUC(G)\cap\RUC(G)$. Conversely, every function of the form $f(g)=\varphi(a,gb)$ is in $\UC(G)$. The algebra $\UC(G)$ is in general strictly contained in $\RUC(G)$ (unless $G$ is a \emph{SIN group}).

On the other hand, suppose $G$ is any first-countable topological group, fix a left-invariant metric $d_L$, and let $L$ be the completion of $G$ with respect to $d_L$. Then, by choosing an appropriate language, one can see $L$ as a metric structure in such a way that $G=\Aut(L)$; see~\cite[\textsection 3]{Mel-hjorth}. When we see $L$ as an isometric $G$-space via the natural left action $G\actson L$, we have a canonical isomorphism between the algebras $\RUC_G^u(L)$ and $\UC(G)$ (the restriction map). This suggests that, for a general structure $M$, the algebra $\RUC_G^u(M)$ is in some sense an analogue of the algebra $\UC(G)$.

From this viewpoint, the conclusion in Proposition~\ref{p:DEF=RUCu} seems natural, and the question was not considered in \cite{ibaDyn} as to whether it was optimal. However, as we observe next, one can actually prove that $\Def(M)=\RUC_G(M)$.

We recall from \cite{bentsa} that $M$ being separably categorical is equivalent  to saying that the action $G\actson M$ is \emph{approximately oligomorphic}, which means that the quotients $M^n{\sslash}G$ are compact for each $n<\omega$ (equivalently, for $n=\omega$). Here, the quotient $M^n{\sslash}G$ is the space of closed orbits $\{\ov{Ga}:a\in M^n\}$ endowed with the metric
$$d(\ov{Ga},\ov{Gb})=\inf_{g\in G}d(a,gb),$$
where we have fixed beforehand some $G$-invariant compatible metric $d$ on $M^n$ (such as the one mentioned earlier). Following model-theoretic terminology, we call the closed orbit $\ov{Ga}$ of a tuple $a\in M^n$ the \emph{type (over $\emptyset$)} of $a$, and we denote it by $\tp(a)$.

\sk 
\begin{prop}\label{p:RUCu=RUC}
If $M$ is $\aleph_0$-categorical, then $\RUC^u_G(M) =\RUC_G(M)$.
\end{prop}
\begin{proof}
Suppose there is $f\in\RUC_G(M)\setminus\RUC_G^u(M)$. Then there exist $\epsilon>0$ and sequences $a_n,b_n\in M$ such that $d(a_n,b_n)\to 0$ and $|f(a_n)-f(b_n)|\geq\epsilon$ for every $n$. Choose an open neighborhood $1\in U\subseteq G$ such that $\sup_{x\in M}|f(gx)-f(x)|<\epsilon/4$ for every $g\in U$. We may assume that $U=\{g\in G:d(gm,m)<\delta\}$ for some finite tuple $m\in M^k$ and some $\delta>0$.

Consider the types $\tp(a_n m)$ and $\tp(b_n m)$ in $M^{1+k}{\sslash}G$. Up to passing to some subsequences, we may assume that they converge to some types $p$ and $q$, respectively. Moreover, since $d(a_n,b_n)\to 0$, we have $p=q$. Now take $c\in M$ and $m'\in M^k$ such that $p=\tp(cm')$. Since $\tp(m')=\tp(m)$, we may assume that $d(m,m')<\delta/2$.

Since $f$ is continuous, there is $\eta>0$ such that $|f(c')-f(c)|<\epsilon/4$ whenever $d(c',c)<\eta$. We may assume that $\eta<\delta/2$. Take $n$ such that $d(\tp(a_n m),p)<\eta$ and $d(\tp(b_n m),p)<\eta$. Thus there exist $g,h\in G$ satisfying $d(g(a_n m),c m')<\eta$ and $d(h(b_n m),cm')<\eta$. In particular, $d(gm,m)<d(gm,m')+d(m',m)<\delta$, so $g\in U$. It follows that $|f(a_n)-f(ga_n)|<\epsilon/4$. Similarly, $|f(b_n)-f(hb_n)|<\epsilon/4$. But then
\begin{align*}
|f(a_n)-f(b_n)| & \leq |f(ga_n)-f(hb_n)|+\epsilon/2 \\
& \leq |f(ga_n)-f(c)|+|f(c)-f(hb_n)|+\epsilon/2 \\
& <\epsilon,
\end{align*}
contradicting that $|f(a_n)-f(b_n)|\geq\epsilon$.
\end{proof}

Let us recall also from \cite{bentsa} that the automorphism groups of separably categorical structures are precisely, up to isomorphism, the \emph{Roelcke precompact} Polish groups, i.e., those Polish groups for which the Roelcke uniformity is totally bounded. In fact, if $G$ is Polish and Roelcke precompact, the action of $G$ on its left completion $L$ is approximately oligomorphic. Hence, when seen as a metric structure, $L$ is indeed separably categorical. We derive the following.

\begin{cor}
Let $G$ be a Roelcke precompact Polish group and let $L$ be its left-completion. Then $\RUC^u_G(L)=\RUC_G(L)$, and so $\UC(G)\simeq\RUC_G(L)$.
\end{cor}

\sk 
The Stone--Gelfand dual of the algebra $\Def(M)$ is known as the \emph{space of 1-types over $M$}, and is denoted by $S_1(M)$. This is the maximal ideal space of the algebra $\Def(M)$, with its canonical compact topology. A maximal ideal (or \emph{1-type}) $p\in S_1(M)$ can also be seen as a function $p\colon \Def(M)\to\R$ that maps each $f\in\Def(M)$ to the unique real number $p(f)$ such that $f-p(f)\in p$. From the logical point of view, a 1-type is a maximal satisfiable set of conditions in a single variable (see, e.g., \cite[\textsection 3]{BenUsv-local-stability}). Indeed, every type $p\in S_1(M)$ can be \emph{realized} in an appropriate elementary extension $M'$ of $M$, in the sense that there is $a'\in M'$ such that for every $f\in\Def(M)$, say given by $f(x)=\varphi(x,b)$, we have $p(f)=\varphi(a',b)$ (as calculated in the extension $M'$). In that case we say that $p$ is the type of $a'$ over $M$, and we write $p=\tp(a'/M)$.

We have proved:

\begin{thm} \label{t:types} 
Every $\aleph_0$-categorical structure $M$ satisfies 
 $\Def(M)=\RUC_G(M)$. Thus the maximal equivariant compactification of the system $\Aut(M)\actson M$ can be identified with the space $S_1(M)$ of 1-types over $M$, and is in particular metrizable.
\end{thm}

For the metrizability, we recall that if $M$ is a separable structure in a countable language, the algebra $\Def(M)$ is separable: it is generated by the functions of the form $x\mapsto\varphi(x,b)$ where $\varphi(x,y)$ varies over the countable set of \emph{finitary, restricted} formulas of $M$ (see \cite[\textsection 6]{BY-et-al}), and $b$ varies over a countable dense set of tuples of $M^{|y|}$ (where $|y|$ is the length of the tuple of variables $y$, which varies with $\varphi$).

Next we point out a useful characterization of the Kat\v{e}tov functions that form the Gromov compactification of $M$ (see Proposition~\ref{p:Gr-comp}.2.b). We recall that if $M$ is $\aleph_0$-categorical and $G=\Aut(M)$, then the left completion $L$ of $G$ can be identified with the semigroup of \emph{elementary self-embeddings} of $M$ (i.e., isometric maps $\sigma\colon M\to M$ satisfying $\varphi(\sigma a)=\varphi(a)$ for every formula $\varphi(x)$ and tuple $a\in M^{|x|}$); see, for instance, \cite[Prop.~2.10]{ibaRando}. Given an element $a\in M$ and an elementary self-embedding $\sigma\colon M\to M$, let $\xi_{a,\sigma}\colon M\to\R$ be the function defined by
$$\xi_{a,\sigma}\colon x\mapsto d(a,\sigma x).$$
Then $\xi_{a,\sigma}$ belongs to the space $\operatorname{K}(M)$ of Kat\v{e}tov functions on $M$.

\begin{prop}\label{p:gM-katetov}
Let $M$ be an $\aleph_0$-categorical structure and $L$ the semigroup of elementary self-embeddings of $M$. Then the Gromov compactification $\g M$, seen as a subset of $\operatorname{K}(M)$, consists precisely of the Kat\v{e}tov functions $\xi_{a,\sigma}$ for $a\in M$ and $\sigma\in L$.
\end{prop}
\begin{proof}
Every type $p\in S_1(M)$ can be realized by some element $a'$ in a separable elementary extension $M'$ of $M$. By $\aleph_0$-categoricity, there is an isomorphism $\sigma'\colon M'\to M$. If $a=\sigma'(a')$ and $\sigma$ is the restriction of $\sigma'$ to $M$, then $a'\in M$, $\sigma\in L$, and 
$p(\varphi_b)=\varphi(a',b)=\varphi(a,\sigma b)$ for every definable predicate $\varphi_b\in\Def(M)$, $\varphi_b(x)=\varphi(x,b)$.

Let $\pi\colon S_1(M)\to\g M$ be the canonical surjection corresponding to the inclusion $\Gro(M)\subseteq \Def(M)$. The image $\pi(p)$ of $p\in S_1(M)$ can be identified with the map $M\to\R$, $b\mapsto p(f_b)$, where $f_b$ is the elementary Kat\v{e}tov function $x\mapsto d(x,b)$. That is, $\pi(p)(b)=p(f_b)=d(a,\sigma b)$, so $\pi(p)=\xi_{a,\sigma}$ as desired.
\end{proof}

In general, arbitrary definable predicates can be difficult to understand. A more tractable algebra is $\Defqf(M)$, the subset of $\Def(M)$ consisting of the functions $x\mapsto\varphi(x,b)$ where $\varphi(x,y)$ is given by a \emph{quantifier-free} formula (i.e., obtained from the basic predicates by continuous combinations and uniform limits, but without quantification). In other words, $\Defqf(M)$ is the closed algebra generated by the \emph{atomic definable predicates}, i.e., the functions of the form
$$x\mapsto P(x,b)$$
where $b\in M^n$ is any tuple and $P\colon M\times M^n\to\mathbb R$ is either the metric of $M$ (thus $n=1$) or one of the distinguished basic predicates of $M$ (possibly precomposed with a reordering and/or repetition of the variable $x$). Hence we have:
$$\Gro(M)\subseteq \Defqf(M)\subseteq \Def(M).$$

In many interesting examples, $M$ has \emph{quantifier elimination}, meaning that every formula is equivalent to a quantifier-free formula (see \cite[\textsection 4.4]{BenUsv-local-stability} or \cite[\textsection 13]{BY-et-al}), and thus, in particular, $\Defqf(M)=\Def(M)$. It is usually easier to check the conditions given in the following definition.

Recall that a partial isometry $p\colon A \to B$ between subsets $A,B\subseteq M$ is a \emph{partial isomorphism} if $P(pa)=P(a)$ for every basic predicate $P(x)$ of $M$ and every tuple $a\in A^{|x|}$. Equivalently, if $\varphi(pa)=\varphi(a)$ for every quantifier-free formula $\varphi(x)$ and $a\in A^{|x|}$.

\begin{defin}\label{d:app-ultrahomogeneity}
A metric structure $M$ is \emph{ultrahomogeneous} if every partial isomorphism $p\colon A \to B$ between finite subsets $A,B\subseteq M$ can be extended to an automorphism of $M$. It is \emph{approximately ultrahomogeneous} if for every such partial isomorphism and every $\eps >0$ there is $g \in \Aut(M)$ such that $d(pa,ga) < \eps$ for all $a \in A$.
\end{defin}

As is well-known, a separably categorical structure has quantifier elimination if and only if it is approximately ultrahomogeneous. Since we could not find a reference, let us recall why this holds. Every $\aleph_0$-categorical structure $M$ is approximately \emph{homogeneous}, meaning that every \emph{elementary} partial map $p\colon A\to B$ (i.e., such that $\varphi(pa)=\varphi(a)$ for every formula $\varphi(x)$ and $a\in A^{|x|}$) between finite subsets can be extended, up to any $\epsilon>0$, to an automorphism of $M$ (see \cite[Cor.~12.11]{BY-et-al}). If moreover every formula is equivalent to a quantifier-free formula, then every partial isomorphism is elementary, and we see that $M$ is ultrahomogeneous. For the converse, by a standard duality argument (and since $M$ is the only separable model of its theory up to isomorphism) it suffices to show that if $a,b\in M^n$ are finite tuples with $\varphi(a)=\varphi(b)$ for every quantifier-free formula $\varphi(x)$, then $\varphi(a)=\varphi(b)$ for every formula. Now the former condition says that $a\mapsto b$ defines a partial isomorphism, and then ultrahomogeneity and the continuity of formulas easily imply the conclusion.

Finally, recall that a \emph{self-embedding} of $M$ is an isometric map $\sigma\colon M\to M$ such that $P(\sigma a)=P(a)$ for every basic distinguished predicate $P(x)$ of $M$ and every tuple $a\in M^{|x|}$. Equivalently, $\varphi(\sigma a)=\varphi(a)$ for every quantifier-free formula $\varphi(x)$ and $a\in M^{|x|}$. Under quantifier elimination, elementary self-embeddings are thus just self-embeddings. In conclusion:

\begin{cor}\label{c:cat-app-ultra}
If $M$ is separably categorical and approximately ultrahomogeneous, then the algebra $\RUC_G(M)$ is generated by the atomic definable predicates of $M$. Moreover, $\g M$ is the space of Kat\v{e}tov functions $\xi_{a,\sigma}\colon x\mapsto d(a,\sigma x)$ for elements $a\in M$ and self-embeddings $\sigma\colon M\to M$.
\end{cor}

Next we review the fundamental examples, including some discrete ones. For the latter, the basic facts mentioned below can be found, for instance, in \cite[\textsection 3.3, \textsection 4.3]{TentZieg}; for the metric examples we give individual references.
All Banach spaces that we consider are over the reals, except in the case of Hilbert spaces, which we consider both over $\R$ and over $\C$.

\sk 
\begin{examples}\label{ex:6Examples}\ 
\begin{enumerate}[wide]\itemsep5pt 
\item Let $X$ be either a countable set with no further structure, or a countable infinite-dimensional vector space over a finite field $F_q$, in each case endowed with the $\{0,1\}$-valued metric. These are $\aleph_0$-categorical, ultrahomogeneous structures. In both cases, the atomic definable predicates boil down to characteristic functions of elements of $X$. Let $G=\operatorname{Sym}(X)$ or $G=\operatorname{GL}_{F_q}(X)$ be the corresponding automorphism group. It follows that $\beta_G$ equals $\g$ and is just the one-point compactification of $X$. For $G=\operatorname{Sym}(X)$, this is the only non-trivial equivariant compactification of~$X$.

\item Let $R$ be the Rado graph---a countable ultrahomogeneous graph containing a copy of every finite graph---, which is $\aleph_0$-categorical. We may endow it with the graph distance: for $R$ this means that if $a\neq b$, $d(a,b)=1$ if $a$ and $b$ are adjacent and $d(a,b)=2$ otherwise. Since the adjacency relation is coded by the metric, we can deduce that the maximal equivariant compactification of $R$ is the Gromov compactification of $(R,d)$.

Moreover, $\beta_G R$ can be identified with $R\cup\{1,2\}^R$, where a base $\mathcal{B}$ for the topology is given as follows: for each $p\in\{1,2\}^A$ defined on a finite subset $A\subseteq R$, let 
$$
B_p=\{b\in R:\forall a\in A,d(a,b)=p(a)\}\cup\{s\in \{1,2\}^R:s\supseteq p\};
$$ 
then $\mathcal{B}$ is the collection of all singletons of elements of $R$, the empty set, and all sets of the form $B_p$.

\item 
The unit sphere $S_{\ell^2}$ of the Hilbert space $\ell^2(\N)$, with the inner product as only distinguished predicate (or, in the complex case, its real and imaginary parts), is $\aleph_0$-categorical and ultrahomogeneous; see \cite[\textsection 15]{BY-et-al}. Since $\|x-b\|=(2-2\operatorname{Re}\langle x,b\rangle)^{1/2}$ (and $\operatorname{Im}\langle x,b\rangle=\operatorname{Re}\langle x,ib\rangle$), we see that $\RUC_G(S_{\ell^2})$ is the closed algebra generated by the functions $x\mapsto \operatorname{Re}\langle x,b\rangle$ for $b\in S_{\ell^2}$. Conversely, by the polarization identity, $\RUC_G(S_{\ell^2})$ is also generated by the functions $x\mapsto \|x-b\|$ for $b\in S_{\ell^2}$. Hence $\Gro(S_{\ell^2})=\RUC_G(S_{\ell^2})$. Taking into account Proposition \ref{p:S-Gr}, this recovers Stoyanov's result mentioned in Example~\ref{ex:Stoyanov}.

\item The Urysohn sphere $\mathbb{U}_1$ is an $\aleph_0$-categorical, ultrahomogeneous metric space; see \cite[\textsection 5]{UsvUrysohn}. It follows that $\Gro(\mathbb{U}_1)= \RUC_G(\mathbb{U}_1)$, giving an alternative proof of Theorem~\ref{t:Urys-sphere}.

Every Kat\v{e}tov function $\xi\in \operatorname{K}(\mathbb{U}_1)$ induces a one-point metric extension $X=\mathbb{U}_1\cup\{a'\}$ of $\mathbb{U}_1$ by setting $d(a',x)=\xi(x)$ for every $x\in \mathbb{U}_1$. By universality of the Urysohn sphere, $X$ embeds in a copy $\mathbb{U}_1'$ of $\mathbb{U}_1$. In other words, we have $\mathbb{U}_1\subseteq X\subseteq \mathbb{U}_1'$ and an isomorphism $\sigma'\colon \mathbb{U}_1'\to \mathbb{U}_1$. Letting $a=\sigma'(a')$ and $\sigma=\sigma'|_{\mathbb{U}_1}$, we see that we can write $\xi$ in the form $\xi_{a,\sigma}\colon x\mapsto d(a,\sigma x)$ of Corollary~\ref{c:cat-app-ultra}. We conclude that $\beta_G\mathbb{U}_1=\operatorname{K}(\mathbb{U}_1)$.

From this description it is easy to produce other non-trivial compactifications of $\mathbb{U}_1$, both proper and non-proper (though necessarily injective, by Proposition~\ref{p:injective}). Indeed, given $0<\lambda<1$, let $\alpha^\lambda$ and $\alpha_\lambda$ be the continuous maps from $\operatorname{K}(\mathbb{U}_1)$ into itself defined by:
$$\alpha^\lambda(\xi)=\max(\lambda,\xi),\ \ \alpha_\lambda(\xi)=\min(\lambda,\xi).$$
Let also $K^\lambda=\alpha^\lambda(\operatorname{K}(\mathbb{U}_1))$ and $K_\lambda=\alpha_\lambda(\operatorname{K}(\mathbb{U}_1))$, and consider the composite maps:
$$\nu^\lambda=\alpha^\lambda\circ\beta^{\mathbb{U}_1}_G\colon \mathbb{U}_1\to K^\lambda,\ \ \nu_\lambda=\alpha_\lambda\circ\beta^{\mathbb{U}_1}_G\colon \mathbb{U}_1\to K_\lambda.$$
Then $\nu^\lambda$ and $\nu_\lambda$ are non-trivial compactifications of $\mathbb{U}_1$, different from $\beta^{\mathbb{U}_1}_G$. Moreover, $\nu^\lambda$ is proper, whereas $\nu_\lambda$ is non-proper. To see the latter, note that (by finite injectivity) for any $x,z_1,\dots,z_n\in\mathbb{U}_1$ there is $y\in\mathbb{U}_1$ such that $d(x,y)=\lambda$ and $d(y,z_i)=\max(\lambda,d(x,z_i))$ for each $i=1,\dots,n$; thus one can construct a sequence $(y_n)$ such that $\nu_\lambda(y_n)\to \nu_\lambda(x)$ but $y_n\notin B_\lambda(x)$ for all~$n$.

\item Now we consider approximately ultrahomogeneous, separably categorical Banach spaces---by which we mean that the unit sphere (or, equivalently, the unit ball) with the induced structure is separably categorical. Other than the Hilbert space, the main examples are the Gurarij space $\mfG$ and the spaces $L^p[0,1]$ for $p\notin 2\N$. (The spaces $L^{2n}[0,1]$ for natural $n>1$ are $\aleph_0$-categorical but not approximately ultrahomogeneous.) For the Gurarij space we refer to \cite{BenHen-gurarij}. For the $L^p$ spaces, $\aleph_0$-categoricity can be deduced from the fact that they are reducts of the \emph{$L^p$ Banach lattices}, which are $\aleph_0$-categorical as per \cite[Fact 17.6]{BY-et-al} (see also \cite[Thm.~3]{Bre-et-al}); for approximate ultrahomogeneity we point to \cite{Ferenc-et-al} and the references therein.
 
Our analysis yields the following.
\end{enumerate}
\end{examples}

\begin{thm}\label{t:Banach-spaces}
Let $V$ be an $\aleph_0$-categorical, approximately ultrahomogeneous Banach space, $G$ its linear isometry group and $S_V$ its unit sphere. Then:
\begin{enumerate}
\item The algebra $\RUC_G(S_V)$ is generated by the family of functions
$$f_v\colon S_V\to\R,\ x\mapsto\|x-v\|$$
for $v\in V$ (not necessarily in $S_V$).
\item The greatest equivariant compactification $\beta_G(S_V)$ is metrizable and can be identified with the space of Kat\v{e}tov functions $\xi\in\operatorname{K}(V)$ of the form
$$\xi_{w,\sigma}\colon V\to\R,\ v\mapsto\|w-\sigma v\|,$$
where $w\in S_V$ and $\sigma\colon V\to V$ is an isometric endomorphism. More precisely, $\xi_{w,\sigma}$ is identified with the unique element of the Stone--Gelfand dual of $\RUC(S_\mfG)$ that maps each $f_v$ to $\xi_{w,\sigma}(v)$.
\item Seeing $\g(S_V)$ as a subset of $\operatorname{K}(S_V)$ as in Proposition~\ref{p:Gr-comp}, the canonical surjection $\beta_G(S_V)\to \g(S_V)$ sends $\xi_{w,\sigma}$ to its restriction to $S_V$.
\end{enumerate}
\end{thm}
\begin{proof}
The basic predicates of $S_V$ as a metric structure are of the form $\|\sum_{i=1}^n\lambda_i x_i\|$ for scalars $\lambda_i$. Hence, the atomic definable predicates boil down to the functions $f_v\colon S_V\to\R,\ x\mapsto\|x-v\|$ for $v\in V$. By Corollary~\ref{c:cat-app-ultra}, this proves our first claim.

Hence a type $p\in S_1(S_V)$ can be identified with the function $V\to\R$, $v\mapsto p(f_v)$. On the other hand, as in the proof of Proposition~\ref{p:gM-katetov}, every $p\in S_1(M)$ can be represented by an element $w\in S_V$ and an embedding $\sigma\colon V\to V$, so that $p(f_v)=\|w-\sigma v\|$. This proves the second claim, and the third is then clear.
\end{proof}

We remark that the Kat\v{e}tov functions $\xi\colon V\to\R$ of the form $\xi=\xi_{w,\sigma}$ as in the theorem are always \emph{normalized} (in the sense that $\xi(0)=1$) and \emph{convex}. We will denote by $\operatorname{K}_C^1(V)$ the compact set of normalized, convex Kat\v{e}tov functions on the Banach space~$V$.

\begin{example}\label{ex:Gurarij}\ 

Let us consider, in more details, the case of the \textit{Gurarij space} $V=\mfG$. Introduced by Gurarij in \cite{Gurarij1966}, $\mfG$ is the unique separable, approximately ultrahomogeneous real Banach space that is universal for finite-dimensional (or separable) normed spaces. See Ben Yaacov and Henson's work \cite{BenHen-gurarij} for a model-theoretically inspired account of the Gurarij space, intended for logicians and non-logicians.

We show first that $\beta_G(S_\mfG)=\operatorname{K}_C^1(\mfG)$. For this we proceed as in the case of the Urysohn sphere. If $V$ is an arbitrary Banach space and $\xi\in\operatorname{K}_C^1(V)$ then, as shown by Ben Yaacov in \cite[Lemma~1.2]{BenYaacov}, one can construct a Banach space $V'$ extending $V$ with a vector $w'\in V'$ such that $\xi(v)=\|w'-v\|$ for all $v\in V$. Now suppose $\xi\in\operatorname{K}_C^1(\mfG)$ and choose $V'$, $w'$ with these properties. The extension $V'$, which we may assume generated by $V$ and $w'$, is separable, so by universality of the Gurarij space it embeds in a copy $\mfG'$ of $\mfG$. We let $\sigma'\colon \mfG'\to\mfG$ be an isomorphism and set $w=\sigma'(w')$ and $\sigma=\sigma'|_{\mfG}$. Then $w\in S_\mfG$, $\sigma$ is an isometric endomorphism of $\mfG$, and $\xi=\xi_{w,\sigma}$. By Theorem~\ref{t:Banach-spaces}, this establishes our claim.

Next we show that $\Gro(S_\mfG)$ is strictly contained in $\RUC_G(S_\mfG)$. For this it suffices to exhibit two distinct $\xi,\xi'\in \operatorname{K}_C^1(\mfG)$ that agree on the sphere. Note that if $h\colon\R_{\geq 0}\to\R$ is a convex, 1-Lipschitz function with $h(0)=1$ and $h(r)\geq r$ for all $r\in\R$, then the formula $\xi(v)=h(\|v\|)$ defines a normalized convex Kat\v{e}tov function on $\mfG$. Hence, for instance, by considering $h(r)=\max\{1,r\}$ and $h'(r)=r+(1-2r)\chi_{[0,1/2]}(r)$, we get two distinct elements $\xi,\xi'\in\beta_G(S_\mfG)$ whose projections to the Gromov compactification coincide, as desired.
\end{example}

We single out the last conclusion, which answers one of the questions set by Pestov in \cite{Pest-Smirnov}.

\begin{thm} \label{t:Gurarij} 
The maximal equivariant compactification of the unit sphere of the Gurarij space is metrizable and does not coincide with its Gromov compactification.
\end{thm}

By Theorem~\ref{t:types}, the metrizability of the greatest equivariant compactification also holds for the unit spheres $S_p$ of the Banach spaces $V_p=L^p[0,1]$ for $1\leq p<\infty$. In the reflexive case $p>1$ ($p\neq 2$), it is asked in \cite{Pest-Smirnov} whether this maximal compactification is given by the natural inclusion $\nu_p\colon S_p\to B_p^w$ of the unit sphere into the unit ball endowed with the weak topology.

Note that if $V$ is a reflexive Banach space and $B_V^w$ denotes its unit ball with the weak topology, then $S_V \to B_V^w$ is a $G$-equivariant compactification (the action of $G=\Aut(V)$ on $B_V^w$ is continuous). On the other hand, if $V$ is a uniformly convex Banach space, then the weak and the norm topologies coincide on the sphere. In particular, for $1<p<\infty$, the map $\nu_p\colon S_p\to B_p^w$ is a proper equivariant compactification of $S_p$.

\begin{thm}
Let $1<p<\infty$, $p\neq 2$. Then the Gromov compactification of $S_p$ is \emph{not} a factor of $\nu_p\colon S_p\to B_p^w$. In particular, $\nu_p$ is neither the maximal equivariant nor the Gromov compactification of $S_p$.
\end{thm}
\begin{proof}
We consider the Banach space $W=L^p([0,2]\times[0,1])$, and we identify $V_p=L^p[0,1]$ with the subspace $V\subseteq W$ of $p$-integrable functions $x\colon [0,2]\times [0,1]\to\R$ with support contained in $[0,1]\times [0,1]$ that are measurable with respect to the first coordinate. That is,
$$V=\{x\in W : x(t_1,t_2)=x'(t_1)\chi_{[0,1]}(t_1)\text{ a.e.\ for some }x'\in L^p[0,1]\}\simeq V_p.$$
Hence we also identify the unit sphere $S_p$ of $V_p$ with the unit sphere $S_V$ of $V$.

The inclusion $V\subseteq W$ (when restricted to the corresponding spheres, or balls, according to the choice of formalization) is an \emph{elementary} embedding of Banach spaces in the sense of continuous logic. Indeed, we can see it as an embedding in the richer language of normed vector lattices. Then, since the theory of atomless $L^p$ Banach lattices has quantifier elimination (see \cite[Fact~17.5]{BY-et-al}), this is an elementary embedding of Banach lattices, and thus also of Banach spaces.

Let $w$ and $w'$ be the elements of $W$ defined as follows:
\begin{align*}
w(t)=\begin{cases}
          1\hspace{15pt}\text{if }t\in [0,1]\times [0,1/2]\\
          -1\hspace{6pt}\text{if }t\in [0,1]\times (1/2,1]\\
          0\hspace{15pt}\text{if }t\in (1,2]\times [0,1]
          \end{cases} &\hspace{10pt}
w'(t)=\begin{cases}
          0\hspace{15pt}\text{if }t\in [0,1]\times [0,1]\\
          1\hspace{15pt}\text{if }t\in (1,2]\times [0,1]
          \end{cases}
\end{align*}
We have $\|w\|=\|w'\|=1$. Since the inclusion $V\subseteq W$ is elementary, we can consider the types $\xi=\tp(w/S_V)$ and $\xi'=\tp(w'/S_V)$ of the structure $V$. In view of Theorem~\ref{t:types}, we can also see $\xi$ and $\xi'$ as elements of the compactification $\beta_G(S_V)$.

Let $\pi_\g\colon\beta_G(S_V)\to \g(S_V)$ and $\pi_\nu\colon\beta_G(S_V)\to B_p^w$ be the canonical projections onto the compactifications $\g S_V$ and $B_p^w$. We claim that $\pi_\g(\xi) \neq \pi_\g(\xi')$ and that $\pi_\nu(\xi) = \pi_\nu(\xi')$. This then implies the theorem.

To see that $\pi_\g(\xi) \neq \pi_\g(\xi')$, consider the vector $v\in S_V$, $v=\chi_{[0,1]\times [0,1]}$, and the function $f_v\in\Gro(S_V)$, $f_v\colon x\mapsto \|x-v\|$. Then $\xi(f_v) = \|w-v\| = (2^p/2)^{1/p} = 2^{\frac{p-1}{p}}$ whereas $\xi'(f_v) = \|w'-v\| = 2^{1/p}$, and these values are distinct since $p\neq 2$.

To see that $\pi_\nu(\xi) = \pi_\nu(\xi')$, we first note that the algebra of RUC functions on $S_p$ corresponding to the compactification $\nu_p$ is generated by the functions $h_z\colon x\mapsto\langle x,z\rangle$, where $z\in L^q[0,1]$, $1/p+1/q=1$, and $\langle x,z\rangle=\int_0^1 xz\, dt_1$ is the canonical pairing. We see the functions $h_z$ as RUC functions on $S_V$. As such, each $h_z$ is a definable predicate of the structure $V$, and thus also of the elementary extension $W$. In fact, as a predicate on~$W$, $h_z$ is given by $h_z(x)=\int x\tilde z\, dt$, where the integral is calculated over $[0,2]\times [0,1]$ and $\tilde z$ is the function $\tilde z(t_1,t_2)=z(t_1)\chi_{[0,1]}(t_1)$. It follows that
$$\xi(h_z) = \int w\tilde z\, dt= 0 = \int w'\tilde z\, dt= \xi'(h_z)$$
for every $z\in L^q[0,1]$. This shows that $\pi_\nu(\xi) = \pi_\nu(\xi')$, and finishes the proof.
\end{proof}

On the other hand, we do not know whether $\beta_G^{S_p}=\g^{S_p}$, but this boils down to the first of the following questions, at least when $p\notin 2\N$.

\begin{question}\ 

Suppose $1\leq p<\infty$, and let $V_p$, $S_p$ and $\nu_p\colon S_p\to B_p^w$ be as above.
\begin{enumerate}
\item Given $v\in V_p$, is the function $S_p\to\R,\ x\mapsto\|x-v\|$ in the algebra $\Gro(S_p)$?

\item Characterize the Kat\v{e}tov functions on $V_p$ that can be represented in the form $v\mapsto \|w-\sigma v\|$ where $\sigma\colon V_p\to V_p$ is an isometric endomorphism and $w\in S_p$.
\item For $p>1$, is $\nu_p$ a factor of the Gromov compactification of $S_p$?
\end{enumerate}
\end{question}

\sk 

Following \cite{Pest-Smirnov}, we have focused on the unit \emph{spheres} of Banach spaces. In the natural examples, the isometric actions on the spheres have the nice feature of being topologically transitive (thus, minimal). However, one may also consider the actions on the unit \emph{balls}, as isometric $G$-spaces. In that case, some problems get easier. If $V$ is a separably categorical, approximately ultrahomogeneous Banach space and $B_V$ is its unit ball (seen as an isometric $G$-space for $G=\Aut(V)$), our analysis shows that to prove $\Gro(B_V)=\RUC_G(B_V)$ it suffices to check that the functions $B_V\to\R,\ x\mapsto\|x-v\|$ for $v\in V$ are in the closed algebra generated by the functions $B_V\to\R,\ x\mapsto\|x-z\|$ for $z\in B_V$. As pointed out to us by I. Ben Yaacov, this is easily verified for the unit balls of the $L^p$ spaces, as follows.

\begin{thm}[Ben Yaacov]
Let $B_p$ be the unit ball of $V_p=L^p[0,1]$ for $1\leq p<\infty$, $p\notin 2\N$. Then the maximal equivariant compactification of $B_p$ is its Gromov compactification.
\end{thm} 
\begin{proof}
Any $v\in V_p$ can be written as $v=\sum_{i=1}^n v_i$ where $v_i\in B_p$ and the supports of the $v_i$ are disjoint (for some given representatives $v_i\colon [0,1]\to\R$). Hence, for every $x\in B_p$ we have $\|x-v\| = \big(\sum_{i=1}^n\|x-v_i\|^p - (n-1)\|x-0\|\big)^{1/p}$, showing that the function $x\mapsto\|x-v\|$ is in $\Gro(B_p)$, as desired.
\end{proof}

For the unit ball $B_\mfG$ of the Gurarij space the situation is different. Note that, proceeding as in Theorem~\ref{t:Banach-spaces} and Example~\ref{ex:Gurarij}, the maximal equivariant compactification of $B_\mfG$ can be identified with the set $\operatorname{K}_C^{\leq 1}(\mfG)$ of convex Kat\v{e}tov functions $\xi\colon \mfG\to\R$ such that $\xi(0)\leq 1$.

\begin{prop}
Let $B_\mfG$ be the unit ball of the Gurarij space $\mfG$. Then $\Gro(B_\mfG)\neq\RUC_G(B_\mfG)$.
\end{prop}
\begin{proof}
As in Example~\ref{ex:Gurarij}, it suffices find two distinct convex, 1-Lipschitz functions $h_1,h_2\colon \R_{\geq 0}\to\R$ such that $h_i(0)\leq 1$, $h_i(r)\geq r$ and $h_1(r)=h_2(r)$ for all $r\leq 1$. We can take $h_1(r)=1+r$ and $h_2(r)=(1+r)\chi_{[0,1)}(r) + 2\chi_{[1,2)}(r) + r\chi_{[2,\infty)}(r)$.
\end{proof}

\sk 

Finally, one can ask about the complexity of the Gromov compactification and of the maximal equivariant compactification of isometric systems in terms of the \emph{dynamical hierarchy} of Banach representations, in the sense of \cite{GM-survey}. Combining the results of \cite{ibaDyn} with our previous analysis, one can see for example that every \emph{tame} function on the Urysohn sphere is constant, and hence that the Gromov compactification of $\mathbb{U}_1$ admits only trivial equivariant representations on Rosenthal Banach spaces. Several results of this kind about dynamical properties of $\beta_GM $ (i.e., $S_1(M)$) for concrete $\aleph_0$-categorical structures $M$ can be deduced from \cite{ibaDyn}.

\noindent\hrulefill

\sk 
\section{On uniform micro-transitivity}\label{s:UMT}
\sk 

A fundamental theorem proved by Effros in \cite{Effros}, sometimes called the Open Mapping Principle or the Effros Microtransitivity Theorem, asserts that if $G\actson X$ is a transitive continuous action of a Polish group $G$ on a Polish space $X$, and $x\in X$ is any point, then the orbit map $g\in G\mapsto gx\in X$ is open. In other words, for every $x\in X$ and every $U\in N_e(G)$ the set $Ux$ is a \nbd of $x$. Ancel \cite{Ancel} coined the term \emph{micro-transitive} for an action with this property.

The theorem actually gives an equivalence.

\begin{thm*}[Effros]
Let $G\actson X$ be a transitive, continuous action of a Polish group on a separable, metrizable space $X$.
The following conditions are equivalent: 
\begin{enumerate}
\item The action is micro-transitive.
\item $X$ is Polish.
\item $X$ is non-meager.
\end{enumerate}
\end{thm*}

In this section we wish to investigate a uniform version of the micro-transitivity property, along with a weak variant. Let us (re)introduce all the definitions. We phrase them in the setting of metric spaces, although they make sense for arbitrary topological or uniform spaces, according to the case.

\begin{defin} Let us say that a continuous action $G\actson X$ of a topological group $G$ on a metric space $(X,d)$ is:
\begin{enumerate}
\item \emph{Micro-transitive}, if for every $x\in X$ and $U\in N_e$ there is $\delta>0$ such that $B_\delta(x)\cap Gx\subseteq Ux$.\footnote{In Ancel's definition the conclusion is $B_\delta(x)\subseteq Ux$ (without intersecting with $Gx$). Hence, the orbits of a micro-transitive action in Ancel's sense are open. We prefer this weaker formulation.}

\item \emph{Weakly micro-transitive}, if for every $x\in X$ and $U\in N_e$ there is $\delta>0$ such that $B_\delta(x)\cap\ov{Gx}\subseteq \ov{Ux}$.

\item \emph{Uniformly micro-transitive (UMT)}, if for every $U\in N_e$ there is $\delta>0$ such that $B_\delta(x)\cap Gx\subseteq Ux$ for all $x\in X$.

\item \emph{Uniformly weakly micro-transitive (UWMT)}, if for every $U\in N_e$ there is $\delta>0$ such that $B_\delta(x)\cap\ov{Gx}\subseteq \ov{Ux}$ for all $x\in X$.
\end{enumerate}
\end{defin}

It is clear that the weak variants are implied by the corresponding strong versions. For actions of Polish groups on Polish spaces the converse is true. In fact, passing from weak micro-transitivity to micro-transitivity is the main step in Ancel's proof of Effros' theorem; see \cite[Lemma~4]{Ancel}. Below we will give an easier, standard argument for this implication \emph{using} Effros' theorem, which we learned from T. Tsankov.

So the main new notion is that of uniform micro-transitivity.\footnote{We observe that an alternative uniform version of micro-transitivity in the context of Banach spaces has been considered recently in \cite{Cab-et-al}. We also point to the work of Kozlov \cite{Kozlov13}, where other related notions are studied.}
Of course, for transitive isometric actions of SIN groups (i.e., containing a basis of conjugation-invariant neighborhoods of the identity), the uniform version is equivalent to the standard one. Note that a topological group is SIN if and only if its left and right uniform structures coincide (see \cite[Prop.~2.17]{RD}).

\begin{examples}\label{ex:UMT}\ 
\begin{enumerate}
\item The group $G=\operatorname{GL}_2(\R)$, being a locally compact Polish group, admits a left-invariant, complete metric $d_L$ (see \cite[Prop.~8.8]{RD}). Let $(X,d)=(G,d_L)$. Then the natural left action $G\actson X$ is a transitive, isometric action of a Polish group on a Polish metric space which is not UMT. Indeed, if it were UMT then the right uniformity on $G$ would be coarser than the uniformity of $d_L$ (see Proposition~\ref{p:UMT-right-unif} below), hence $G$ would be SIN, which is not (see \cite[p.~45--46]{RD}).

\item The system $U(\ell^2)\actson S_{\ell^2}$ is uniformly micro-transitive. In fact, it is not difficult to see that the action is micro-transitive for the finer topology on $U(\ell^2)$ given by the operator norm, which is SIN (but not Polish). On the other hand, the diagonal actions $U(\ell^2)\actson (S_{\ell^2})^n$ for $n\geq 2$, or the action on the unit ball $U(\ell^2)\actson B_{\ell^2}$, are not UMT.

\item The action $\Iso(\mathbb{U}_1)\actson \mathbb{U}_1$ is UMT, and so are all the diagonal actions  $\Iso(\mathbb{U}_1)\actson (\mathbb{U}_1)^n$ for $n<\omega$. Moreover, the actions $(\Iso(\mathbb{U}_1),d_u)\actson (\mathbb{U}_1)^n$ are UMT, where $d_u$ denotes the metric of uniform convergence (defined by $d_u(g,h)=\sup_{z\in\mathbb{U}_1}d(gz,hz)$), which is bi-invariant (thus SIN) and refines the topology of pointwise convergence. This is implied by the following proposition.
\end{enumerate}
\end{examples}

\begin{prop}
Denote $G=\Iso(\mathbb{U}_1)$. Suppose $\epsilon>0$ and $x,y\in(\mathbb{U}_1)^n$ are such that $Gx=Gy$ and $d(x_i,y_i)\leq\epsilon$ for each $i=1,\dots,n$. Then there is $g\in G$ such that $gx_i=y_i$ for each $i=1,\dots,n$ and $d(gz,z)\leq\epsilon$ for every $z\in\mathbb{U}_1$.
\end{prop}
\begin{proof}
Fix $\eps$, $n$ and $x_i,y_i$ as in the statement. We claim that for a given $z\in\mathbb{U}_1$ we can find $g\in G$ such that $gx_i=y_i$ for each $i$ and $d(gz,z)\leq\epsilon$. Then the proposition follows from a standard back-and-forth argument.

By finite injectivity and ultrahomogeneity, to prove the claim it suffices to show that the finite metric space $F=\{x_1,\dots,x_n,y_1,\dots,y_n,z\}$ can be extended to a metric space $F'=F\cup\{z'\}$ of diameter at most 1 such that $d(z',y_i)=d(z,x_i)$ for each $i$, and $d(z',z)\leq\epsilon$. Such an extension can be obtained by setting:
\begin{equation*}
\begin{cases}
d(z',y_i)=d(z,x_i)\\ 
d(z',x_i)=\min\{1,d(z,x_i)+d(x_i,y_i):i=1,\dots,n\}\\ 
d(z',z)=\min\{\eps,d(z,x_i)+d(y_i,z):i=1,\dots,n\}. 
\end{cases}
\end{equation*}
A one-by-one inspection shows that all triangles inequalities are satisfied.
\end{proof}

\sk

We have already considered uniform weak micro-transitivity in Section~\ref{s:Urysohn}, were it proved useful to establish $\Gro(X)=\RUC_G(X)$ for Urysohn-like spaces. On the other hand, in Section~\ref{s:Models} we pointed to the importance of the intermediate equation $\RUC^u_G(X)=\RUC_G(X)$, where $\RUC^u_G(X)$ is the algebra of RUC functions that are uniformly continuous with respect to the metric of $X$. We now observe the following.

\begin{prop}\label{prop:RUCu(M)=RUC(M)}
Let $(X,d)$ be a metric space and let $G\actson X$ be a continuous isometric action. Suppose the action is topologically transitive and UWMT. Then $\RUC^u_G(X)= \RUC_G(X)$.
\end{prop}
\begin{proof}
We have $\ov{Gx}=X$ for all $x$. If $f\in \RUC_G(X)$, a straightforward combination of the property defining RUC, the uniform condition $B_\delta(x)\subseteq \ov{Ux}$, and the continuity of $f$, yields that $f$ is uniformly continuous with respect to the metric of $X$.
\end{proof}

In Section~\ref{s:Models} we showed that every separably categorical structure $M$ under the action of its automorphism group satisfies $\RUC^u_G(M)=\RUC_G(M)$. One may then ask whether every topologically transitive $\aleph_0$-structure is UWMT.

The answer is negative, because, as we mentioned before, a weakly micro-transitive Polish action is necessarily micro-transitive, and for isometric systems this implies in turn that all orbits are closed (see Lemma~\ref{l:generic-vs-weakly-generic} below). Thus, for instance, a topologically transitive, non-transitive $\aleph_0$-categorical structure (e.g., the unit sphere of the Gurarij space or of the $L^p$ spaces for $p\neq 2$) cannot be UWMT. Nevertheless, we will see that this is the only obstruction: a transitive $\aleph_0$-categorical structure is uniformly micro-transitive.

Given a system $G\actson X$, let us say that a point $x\in X$ is \emph{(weakly) generic} if for every $U\in N_e$ there is $\delta>0$ such that $B_\delta(x)\cap Gx\subseteq Ux$ (respectively, $B_\delta(x)\cap\ov{Gx}\subseteq \ov{Ux}$).

\begin{lem}\label{l:generic-vs-weakly-generic}
Let $G\actson X$ be a continuous action of a Polish group on a Polish metric space. Given $x\in X$, the following are equivalent:
\begin{enumerate}
\item $Gx$ is non-meager in its closure.
\item $Gx$ is comeager in its closure.
\item $x$ is generic.
\item $x$ is weakly generic.
\end{enumerate}
In particular, if the action is isometric and (weakly) micro-transitive then every orbit is closed.
\end{lem}
\begin{proof}
The implications $(1)\Rightarrow (2)\Rightarrow (3)$ follow from Effros' theorem, and $(3)\Rightarrow (4)$ is clear. We show $(4)\Rightarrow (1)$. Suppose that $x$ is weakly generic and $Gx\subseteq\bigcup_{n \in \N } F_n$, where the $F_n\subseteq\ov{Gx}$ are closed. Consider the orbit map $\pi\colon G\to X$, $g\mapsto gx$. Then $G=\bigcup_{n \in \N }\pi^{-1}(F_n)$, and as $G$ is Polish there are $n$ and an open set $U\subseteq G$ such that $U\subseteq\pi^{-1}(F_n)$. Hence $Ux\subseteq F_n$, and by weak genericity of $x$ we see that $F_n$ has non-empty interior relative to $\ov{Gx}$. This shows that $Gx$ is non-meager in $\ov{Gx}$.
 
If the action is isometric and weakly micro-transitive, and $y\in\ov{Gx}$, then $Gx$ and $Gy$ are comeager subsets of $\ov{Gx}=\ov{Gy}$, whence $Gx=Gy$. Thus every orbit is closed.
\end{proof}

\sk 
\begin{prop}
Let $G$ and $X$ be Polish and $G\actson X$ be a continuous action. If the action is UWMT, then it is UMT.
\end{prop}
\begin{proof}
If the action is UWMT then every point is weakly generic, and in fact generic by the previous lemma. Let $U\in N_e$ and choose $V\in N_e$ such that $V^{-1}V\subseteq U$. By UWMT, there is $\delta>0$ such that $B_\delta(x)\cap\ov{Gx}\subseteq\ov{Vx}$ for every $x\in X$. Now let $x\in X$ be arbitrary and take $y\in B_\delta(x)\cap Gx$. Since $y$ is generic, there is $\delta'>0$ such that $B_{\delta'}(y)\cap Gy\subseteq Vy$. Since $y\in \ov{Vx}$, there is $v\in V$ such that $d(y,vx)<\delta'$, and as $y\in Gx$, we have $vx\in B_{\delta'}(y)\cap Gy$. Hence there is $v'\in V$ with $vx=v'y$, so $y\in V^{-1}Vx\subseteq Ux$. We conclude that $B_\delta(x)\cap Gx\subseteq Ux$, and that the action is UMT.
\end{proof}

From now on we will concentrate on transitive systems. If $G$ is a topological group, $G\actson X$ is a transitive action and $x\in X$ is any point, we can consider the quotient uniformity on $X$ induced by the orbit map $(G,\U_R)\to X$, $g\mapsto gx$. A basis of entourages is given by the sets $\{(gx,ugx):g\in G,u\in U\}$ where $U\in N_e(G)$ (see \cite[p.~128]{Pe-nbook}), which shows that the quotient uniformity does not depend on the choice of the point $x$. We call it the \emph{right uniformity} on $X$ and denote it by $\U_R^X$. If the action $G\actson X$ is continuous and $G$ and $X$ are Polish, then the right uniformity on $X$ is compatible. We may also remark that a bounded function $f\colon X\to\R$ is in $\RUC_G(X)$ precisely if it is uniformly continuous with respect to $\U_R^X$. Hence, the compact replica of $\U_R^X$ is the uniformity $\U_G$ defined in Remark~\ref{r:prox-delta-G-X}.

We recall that a map $\pi\colon (Y,\U_Y)\to (Z,\U_Z)$ between uniform spaces is \emph{uniformly open} if for every entourage $\eps\in \U_Y$ there is $\delta\in \U_Z$ such that $B_\delta(\pi(y))\subseteq \pi(B_\eps(y))$ for every $y\in Y$---where $B_\eps(y)=\{y'\in Y:(y,y')\in\eps\}$, and similarly for $B_\delta(z)$.

\begin{prop}\label{p:UMT-right-unif}
Let $G\actson X$ be a transitive, continuous action of a Polish group $G$ on a Polish metric space $(X,d)$. The following conditions are equivalent:
\begin{enumerate}
\item The action is UMT.
\item For any $x\in X$, the orbit map $(G,\U_R)\to (X,d),\ g\mapsto gx$ is uniformly open.
\item The right uniformity on $X$ is coarser than the uniformity of $d$.
\item Every $G$-equivariant compactification $\nu\colon X\to K$ is $d$-uniform.
\item $\RUC_G^u(X)=\RUC_G(X)$.
\end{enumerate}
\end{prop}
\begin{proof}
The basic entourages of $(G,\U_R)$ are of the form $\eps_U=\{(g,ug):g\in G,u\in U\}$ for $U\in N_e$. Hence $(B_{\eps_U}(g))x=Ugx$, and we see that (2) is just a rephrasing of UMT, because the action is transitive. UMT can also be phrased as saying that for every $U\in N_e$ there is $\delta>0$ such that
$$\{(y,z)\in X^2:d(y,z)<\delta\}\subseteq \{(gx,ugx)\in X^2:g\in G,u\in U\},$$
 which is precisely (3). On the other hand, (3) implies that the $G$-compactifications of $X$ (which are always right uniformly continuous) are $d$-uniform, which in turn implies (5) because RUC functions factor through equivariant compactifications.

Finally, if $\RUC_G^u(X)=\RUC_G(X)$ then the compact replica $\U_G$ of $\U_R^X$ is coarser than the uniformity of $d$. On the other hand, a metrizable uniformity on a set $Z$ is finer than any other uniformity on $Z$ with a coarser precompact replica; see \cite[p.~27]{Isbell}. Hence the right uniformity on $X$ is coarser than the uniformity of $d$.
\end{proof}

We can deduce our claim about $\aleph_0$-categorical structures.

\begin{prop} \label{p:TupleOrbitCase}
Let $M$ be a separably categorical structure, $G$ its automorphism group, and let $a\in M^n$ be a tuple such that the orbit $Ga$ is closed. Then the action $G\actson Ga$ is uniformly micro-transitive.
\end{prop}
\begin{proof}
Follows from Propositions \ref{p:UMT-right-unif} and \ref{p:RUCu=RUC}.
\end{proof}

We realize \emph{a posteriori} that this proposition is essentially equivalent to a result of Ben Yaacov and Usvyatsov, which is phrased in purely model-theoretic terms: see Proposi\-tion~2.9 of \cite{BenUsv}.

In fact, a slightly stronger result holds, which can be seen as a uniform Effros' theorem for isometric actions of Polish Roelcke precompact groups.

\begin{thm} \label{t:UnifEffr} 
Let $G$ be a Polish Roelcke precompact group. Then every transitive, isometric, continuous action $G\actson X$ on a Polish metric space $(X,d)$ is uniformly micro-transitive.
\end{thm}

This can be deduced from Proposition \ref{p:TupleOrbitCase} by seeing the space $X$ as a metric imaginary sort of an $\aleph_0$-categorical structure $M$ such that $G=\Aut(M)$. Nevertheless, we provide a self-contained topological argument.

\begin{proof}
Let $d_L$ be a compatible left-invariant metric on $G$ and let $L$ be the completion of $G$ with respect to $d_L$. We will use the fact that if $G$ is PRP, then the action $G\actson L$ is approximately oligomorphic (see the discussion and references in Section~\ref{s:Models}).

Let $G\actson X$ be as in the statement and let us fix a point $x\in X$. Since the action is isometric, the orbit map $G\to X$, $g\mapsto gx$ is left uniformly continuous and hence extends continuously to the map $L\to X$, $\xi\mapsto \xi x$.

If the action is not UMT, there are $U\in N_e$ and $y_n,z_n\in X$ such that $d(y_n,z_n)\to 0$ but $y_n\notin Uz_n$ for all $n$. We may assume $U$ is of the form $U=\{g\in G:d_L(g,e)<\epsilon\}$ for some $\epsilon>0$. Also, we can write $y_n=g_nx$ and $z_n=h_nx$ for some $g_n,h_n\in G$.

Now, since the action $G\actson L$ is approximately oligomorphic, the quotient space $L^3{\sslash}G$ is compact. Thus by considering the sequence $(g_n,h_n,e)\in L^3$, we see that there are $f_n\in G$ and $\xi,\zeta,\chi\in L$ such that
$$(f_ng_n,f_nh_n,f_n)\to (\xi,\zeta,\chi).$$
It follows that $f_ny_n\to \xi x$ and $f_nz_n\to \zeta x$. Since $d(y_n,z_n)\to 0$ and the metric is $G$-invariant, we have $\xi x=\zeta x$.

Let $V=\{g\in G:d_L(g\chi,\chi)<\epsilon/4\}$, which is open. By Effros' theorem, there is $\delta>0$ such that $B_\delta(\xi x)\subseteq V\xi x$. Now let $n$ be large enough that $f_ny_n\in B_\delta(\xi x)$, $f_nz_n\in B_\delta(\xi x)$ and $d_L(f_n,\chi)<\epsilon/4$. Hence there are $v,w\in V$ such that $f_ny_n=v\xi x$ and $f_nz_n=w\xi x$. Letting $u=f_n^{-1}vw^{-1}f_n$, we have $y_n=uz_n$ and
\begin{align*}
d_L(u,e) & = d_L(vw^{-1}f_n,f_n) \leq d_L(vw^{-1}f_n,vw^{-1}\chi) + d_L(vw^{-1}\chi,\chi) + d_L(f_n,\chi)\\
& \leq 2d_L(f_n,\chi) + d_L(vw^{-1}\chi,v\chi) + d_L(v\chi,\chi)
<\epsilon.
\end{align*}
Hence $y_n\in Uz_n$, a contradiction.
\end{proof} 

For more about transitive isometric actions of Polish Roelcke precompact groups, see Ben Yaacov's article \cite[\textsection 5]{BY-transitive}.

Two concrete cases of the previous results are given by the Examples~\ref{ex:UMT}.2 and \ref{ex:UMT}.3 discussed above. In fact, all examples that we know of $\aleph_0$-categorical structures $M$ on which the automorphism group $G$ acts transitively have the stronger property that the action $(G,d_u)\actson M$ is micro-transitive, where $(G,d_u)$ denotes the group $G$ endowed with the (bi-invariant) metric $d_u$ of uniform convergence: $d_u(g,h)=\sup_{x\in M}d(gx,hx)$. Note that $d_u$ induces the coarsest (up to uniform equivalence) bi-invariant metric on $G$ that refines the topology of pointwise convergence (see \cite[\textsection 2]{BenBerMel}). The following problem, raised by T. Tsankov, is also closely related to the questions considered by I. Ben Yaacov in \cite[\textsection 4]{BY-transitive}.

\begin{question}[Tsankov]\ 

Let $M$ be an $\aleph_0$-categorical structure such that the action $G\actson M$ of its automorphism group is transitive. Is the action $(G,d_u)\actson M$ necessarily micro-transitive?
\end{question}

\noindent\hrulefill


\sk 

\bibliographystyle{amsplain}

\begin{thebibliography}{100}

\bibitem{Akin} E. Akin, \emph{Recurrence in topological dynamics: Furstenberg
	families and Ellis actions}, University Series in Mathematics, New
York, Plenum Press, 1997.


\bibitem{Ancel} F.D. Ancel, \emph{An alternative proof and applications of a theorem of E. G. Effros}, Michigan Math. J.  34  (1987),  no. 1, 39--55.

\bibitem{BenYaacov} I. Ben Yaacov, \emph{The linear isometry group of the Gurarij space is universal}, Proc. Amer. Math. Soc. 142 (2014), no. 7, 2459-2467.

\bibitem{BY-transitive} I. Ben Yaacov, \emph{On a Roelcke-precompact Polish group that cannot act transitively on a complete metric space}, Israel J. Math.  224  (2018),  no. 1, 105--132.

\bibitem{BY-et-al} I. Ben Yaacov, A. Berenstein, C.W. Henson and A. Usvyatsov, \emph{Model theory for metric structures}, Model theory with applications to algebra and analysis. Vol. 2, 
 315--427, London Math. Soc. Lecture Note Ser., 350, Cambridge Univ. Press, Cambridge,  2008.

\bibitem{BenHen-gurarij} I. Ben Yaacov and C.W. Henson, \emph{Generic orbits and type isolation in the Gurarij space}, Fund. Math.  237  (2017),  no. 1, 47--82.

\bibitem{BenBerMel} I. Ben Yaacov, A. Berenstein and J. Melleray, \emph{Polish topometric groups}, Trans. Amer. Math. Soc.  365  (2013),  no. 7, 3877--3897.

\bibitem{bentsa} I. Ben Yaacov and T. Tsankov, \emph{Weakly almost periodic functions, model-theoretic stability, and minimality of topological groups}, Trans. Amer. Math. Soc., \textbf{11} (2016), 8267--8294.

\bibitem{BenUsv} I. Ben Yaacov and A. Usvyatsov, \emph{On d-finiteness in continuous structures}, Fund. Math.  194  (2007),  no. 1, 67--88.

\bibitem{BenUsv-local-stability} I. Ben Yaacov and A. Usvyatsov, \emph{Continuous first order logic and local stability}, Trans. Amer. Math. Soc.  362  (2010),  no. 10, 5213--5259.

\bibitem{Bre-et-al} J. Bretagnolle, D. Dacunha-Castelle and J-L. Krivine, \emph{Lois stables et espaces Lp}, (French) Ann. Inst. H. Poincaré Sect. B (N.S.)  2  1965/1966 231--259.

\bibitem{Br}
R.B. Brook, {\it A construction of the greatest ambit}, Math. Systems Theory, {\bf 6} (1970), 243--248.

\bibitem{Cab-et-al}
F. Cabello S\'anchez et al., {\it On Banach spaces whose group of isometries acts micro-transitively on the unit sphere},  J. Math. Anal. Appl.  488  (2020), no. 1, 124046, 14 pp.

\bibitem{DPS} D. Dikranjan, Iv. Prodanov and L. Stoyanov, \emph{Topological groups: characters, dualities and minimal group topologies}, Pure
and Appl. Math. 130, Marcel Dekker, New York-Basel, 1989.


\bibitem{Effros} E.G. Effros, 
\emph{Transformation groups and C$^*$-algebras}, 
 Ann. of Math. (2)  81  (1965), 38--55.

\bibitem{Ferenc-et-al} V. Ferenczi, J. Lopez-Abad, B. Mbombo and S. Todorcevic, \emph{Amalgamation and {R}amsey properties of $L_p$ spaces}, Adv. Math. 369 (2020), 107190, 76 pp.

\bibitem{GM-survey}
E. Glasner and M. Megrelishvili,
\emph{Representations of dynamical systems on Banach spaces},
in: Recent Progress in General Topology III, 
(Eds.: K.P. Hart, J. van Mill, P. Simon),  
Springer-Verlag, Atlantis Press, 2014, 399--470. 

\bibitem{Gurarij1966} V.I. Gurariĭ, \emph{Spaces of universal placement, isotropic spaces and a problem of Mazur on rotations of Banach spaces} (Russian),  Sibirsk. Mat. Ž. 7 (1966), 1002--1013.


\bibitem{ibaDyn} T. Ibarluc\'{i}a, \emph{The dynamical hierarchy for {R}oelcke precompact {P}olish groups}, Israel J. Math., \textbf{2} (2016), 965--1009.

\bibitem{ibaRando} T. Ibarluc\'{i}a, \emph{Automorphism groups of randomized structures}, J. Symb. Log.  82  (2017),  no. 3, 1150--1179.

\bibitem{Isbell}
J.R. Isbell, \textit{Uniform spaces}, Math. Surveys 12, AMS, 1964 (Reprinted 1986). 


\bibitem{Kozlov13}
K.L. Kozlov, \textit{Spectral decompositions of spaces induced by spectral decompositions of acting groups}, Top. Appl., 160 (2013) 1188--1205.

\bibitem{Lud-Vr80}
H. Ludescher and J. de Vries, {\it A sufficient condition for the
	existence of a $G$-compactification}, Proc. Kon. Neder. Akad. Wet.
Ser. A, {\bf 83} (1980), 263--268.

\bibitem{Me-Ex88}
M. Megrelishvili,
{\it A Tychonoff G-space not admitting a compact Hausdorff G-extension or a G-linearization},
Russian Math. Surveys {\bf 43:2} (1988), 177--178.


\bibitem{Me-opit07}
M. Megrelishvili,
\emph{Topological transformation groups: selected topics}. Survey paper in: 
Open Problems in Topology II (Elliott Pearl, editor), Elsevier Science (2007), pp. 423--438.

\bibitem{Me-c_018} 
M. Megrelishvili, {\it A note on the topological group $c_0$}, Axioms-372065 (2018).

\bibitem{Mel-07}
J. Melleray, 
\textit{On the geometry of Urysohn's universal metric space}, Topology and its Applications 154 (2007) 384–-403. 

\bibitem{Mel-hjorth} J. Melleray, 
\textit{A note on Hjorth's oscillation theorem}, J. Symbolic Logic 75 (2010), no. 4, 1359--1365.


\bibitem{NW}
S.A. Naimpally and B.D. Warrack, \emph{Proximity Spaces}, Cambridge University Press, 1970.


\bibitem{Pe-nbook}
V. Pestov, \textit{Dynamics of infinite-dimensional groups. The Ramsey-Dvoretzky-Milman phenomenon,} University Lecture Series, vol. {\bf40}, AMS (2006).


\bibitem{Pest-Smirnov}
V. Pestov, \textit{A topological transformation group without non-trivial equivariant compactifications}, Advances in Mathematics 311 (2017) 1--17. 


\bibitem{RD}
W. Roelcke and S. Dierolf, {\em Uniform structures on topological groups and their quotients\/}, McGraw-Hill (1981).


\bibitem{Sm-geom} Yu.M. Smirnov, \emph{Can simple geometric objects be maximal compact extensions for $\R^n$?}, Russian Math. Surveys,
\textbf{49} (1994), no. 6, 214--215.

\bibitem{Sto} L. Stoyanov, \emph{On the infinite-dimensional unitary
	groups}, C.R. Acad. Bulgare Sci., 36 (1983), 1261--1263.

\bibitem{Sto2} L. Stoyanov, \emph{Total minimality of the unitary groups}, Math. Z. 187 (1984), no. 2, 273--283.

\bibitem{TentZieg} K. Tent and M. Ziegler, \emph{A course in model theory}, Lecture Notes in Logic, 40. Association for Symbolic Logic, La Jolla, CA; Cambridge University Press, Cambridge,  2012.


\bibitem{Urys} 
P. Urysohn, \textit{Sur un espace m{\'e}trique universel}, C. R. Acad. Sci. Paris 180 (1925), 803--806; Bull.
Sci. Math. 51 (1927), 43--64 and 74--90.


\bibitem{Usp90} 
V.V. Uspenskij, 
\textit{On the group of isometries of the Urysohn universal metric space}, Comment. Math. Univ. Carolinae 31 (1990), 181--182.


\bibitem{Usp-comp}
V.V. Uspenskij, {\em Compactifications of topological groups},
Proceedings of the Ninth Prague Topological Symposium (Prague,
August 19--25, 2001). Edited by Petr Simon. Published April 2002
by Topology Atlas (electronic publication), 331--346. ArXiv
e-print, math.GN/0204144.

\bibitem{Usp-sub} V.V. Uspenskij, \emph{On subgroups of minimal topological groups}, Topology and Appl., 155:14 (2008), 1580--1606.

\bibitem{UsvUrysohn} A. Usvyatsov, \emph{Generic separable metric structures}, Topology Appl. 155 (2008), no. 14, 1607--1617.
 
 
\bibitem{Vr-can75} J. de Vries, \emph{Can every Tychonoff $G$-space equivariantly be embedded in a compact Hausdorff $G$-space?}, Math. Centrum 36, Amsterdam, Afd. Zuiver Wisk. (1975).


\bibitem{Vr-loccom78} J. de Vries, \emph{On the existence of $G$-compactifications}, Bull. Acad. Polon. Sci. ser. Math., \textbf{26} (1978), 275--280.


\end{thebibliography}

\end{document}